\documentclass{article}
\usepackage[utf8]{inputenc}

\usepackage{bbm}
\usepackage[overload]{empheq}
\usepackage{amsfonts}
\usepackage{amsmath}
\usepackage{mathtools}
\usepackage{amssymb}
\usepackage{dsfont}
\usepackage{diagbox}
\usepackage{graphicx}
\usepackage{bm}
\usepackage{tikz, pgfplots}
\usetikzlibrary{automata,positioning,arrows, intersections, calc}
\pgfplotsset{compat=1.8}
\usepackage{caption}
\usepackage{subcaption}
\usepackage{placeins, comment}
\usepackage{latexsym,amsthm,xcolor,graphicx, longtable}
\usepackage{hyperref}
\usepackage{mathtools}
\usepackage{tikz, pgfplots}
\usetikzlibrary{automata,positioning,arrows, intersections, calc}
\pgfplotsset{compat=1.8}

\DeclareMathOperator*{\argmax}{arg\,max}
\DeclareMathOperator*{\argmin}{arg\,min}

\newtheorem{thm}{Theorem}
\newtheorem{lemma}{Lemma}

\newtheorem{prop}{Proposition}
\newtheorem{corollary}{Corollary}

\theoremstyle{definition}
\newtheorem{remark}{Remark}
\newtheorem{definition}{Definition}[section]

\newtheorem{assumption}{Assumption}[section]
\newtheorem{example}{Example}[section]

\def\S{\mathcal S}
\def\C{\mathcal C}
\def\Re{\mathcal R}

\def\R{\mathbb{R}}
\def\Z{\mathbb{Z}}

\def\P{{\mathbb P}}     
\def\E{{\mathbb E}}

\def\Sp{\mathcal S}
\def\C{\mathcal C}
\def\Re{\mathcal R}
\def\R{\mathbb R}

\def\Z{\mathbb Z}

\def\dlim{\displaystyle \lim}
\newcommand{\TV}{\text{TV}}

\definecolor{ao(english)}{rgb}{0.0, 0.5, 0.0}

\usepackage{authblk}

\title{Boundary-induced slow mixing for Markov chains and its application to stochastic reaction networks}

\author[1,2]{Wai-Tong (Louis) Fan}
\author[3]{Jinsu Kim}
\author[1]{Chaojie Yuan}

\affil[1]{Department of Mathematics, Indiana University,  IN, USA}
\affil[2]{Department of Organismic and Evolutionary Biology, Harvard University,  MA, USA}
\affil[3]{Department of Mathematics, POSTECH, South Korea}

\date{\today}

\begin{document}

\maketitle

\begin{abstract}
Markov chains on the non-negative quadrant of dimension $d$ are often used to model the stochastic dynamics of the number of $d$ entities, such as  $d$ chemical species in stochastic reaction networks. The infinite state space poses technical challenges, and
the boundary of the quadrant can have a dramatic effect on the long term behavior of these Markov chains. For instance,  the boundary can slow down the convergence speed of an ergodic Markov chain towards its stationary distribution due to the extinction or the lack of an entity.
In this paper, we quantify  this slow-down for a class of stochastic reaction networks and for more general Markov chains on the non-negative quadrant.
We establish general criteria for such a Markov chain to exhibit a power-law lower bound for its mixing time. The lower bound
is of order $|x|^\theta$ for all initial state $x$ on a boundary face of the quadrant, 
where  $\theta$ is characterized by the local behavior of the Markov chain near the boundary of the quadrant. A better understanding of how these lower bounds   arise leads to insights into how the structure of chemical reaction networks contributes to slow-mixing.
\end{abstract}

\section{Introduction}

It is known that steady states, or stationary distributions, is 
 a key concept in the analysis and application of Markov chains, providing essential information about the long-term behavior, performance, and optimization of systems modeled by stochastic processes. For example, researchers often design engineering systems and Markov Chain Monte Carlo (MCMC) methods to reach a target stationary distribution efficiently, or use  experimentally-measured  data over a feasibly long period of time as a proxy of the steady states \cite{gupta2014comparison, ekman1963transport}. However, if the system takes too long to reach the stationary distributions, then observed data under feasible time-scale may no longer reflect the stationary behavior of the system. Understanding \textit{how fast}
 Markov chains reach their  stationary distributions (when they exist) is crucial in scientific applications, from the development of MCMC methods and randomized algorithms 
to the design and control of engineered systems such as
bio-chemical reaction networks.

A common way to quantify how fast a Markov chain reaches a stationary distribution is through the concept of \textit{mixing time}. Let $X=(X(t))_{t\in [0,\infty)}$ be a continuous-time Markov chain which 
admits a stationary distribution $\pi$. 
Let $\delta\in (0,1)$ be a fixed number that is often chosen by the user of the Markov chain. The mixing time of  $X$ starting at $x$ with threshold $\delta$, denoted by $t^{\delta}_{\rm mix}(x)$, is the minimum time it takes for the total variation distance between the law of $X(t)$ and the stationary distribution to be smaller than  $\delta$. That is,
 \begin{align}\label{E:mixing}
t^{\delta}_{\rm mix}(x) := \inf \{t \ge 0 : \Vert P^t(x,\,\cdot\,) - \pi \Vert_{\rm TV} \le \delta\},
\end{align}
where $P^t(x,\,\cdot\,)$ is the probability distribution of the process at time $t$ with initial condition $x$ and where the total variation distance between two probability measures on a measurable space $(\mathbb{S},\mathcal{F})$ is defined as $\Vert \mu - \nu\Vert_{\TV} := \sup_{A\in \mathcal{F}}|\mu(A) - \nu(A)|$.  When the space $\mathbb{S}$ is discrete, $\|\mu-\nu\|_{\TV} = \frac12 \sum_{x\in \mathbb{S}} |\mu(x) - \nu(x)|$. In this paper, the state space of our Markov chain is always an infinite subset of $\Z_+^d$.

Studies on mixing times of Markov chains often focused on establishing upper bounds.
For the case of a finite state space, a common interest is the asymptotic behavior of the mixing times with respect to the number of states. For example, when a deck of \(n\) cards is mixed in a top-to-random manner, at least \(O(n \ln n)\) shuffles are required to completely mix the deck uniformly for all initial conditions \cite{aldous1986shuffling, levin2017markov}. For a countably infinite state space, mixing times are often studied as functions of initial conditions. For instance, for a continuous-time Markov chain modeling the copy numbers of interacting chemical species defined in \(\mathbb{Z}^d_{\ge 0}\),  the growth rate of the mixing times as a function of the initial amount of the species has been studied \cite{xu2023full,mixing_And_Kim,anderson2023new}.

While an upper bound of the mixing time does not tell us  
a definitive minimum time required for the Markov chain to reach near its stationary distribution, a lower bound will. In this sense, a lower bound can indicate how ``slowly" the Markov chain mixes.
Indeed, many biochemical systems are known or believed to converge to their stationary states slowly for various reasons. For example, the shortage of reactants can induce slow convergence to the stationary distribution in autocatalytic reaction systems \cite{awazu2010discreteness}. Another cause of slow convergence is the existence of local minima on the energy landscape, as seen in the protein folding process. Protein folding is often characterized by a rugged energy landscape with multiple local minima, causing the process to become trapped in these states for extended periods \cite{bryngelson1987spin, subramanian2020slow}. 
Knowing the minimum time to reach stationary is also important for the design of MCMC algorithms. In \cite{vialaret2020convergence}, the authors showed that
some non-reversible MCMC algorithms have the undesirable property to slow down the convergence of the Markov chain, a point which has been overlooked by the literature.

Understanding these and other ``slow-mixing" behavior  is critical for applying stochastic models in practical fields such as synthetic biology and bioinformatics,  because  understanding what and how things can go wrong (or slow) can significantly affect the design and the accurate interpretation of engineered systems. 
This paper contributes to the mathematical study of quantifying slow-mixing behaviors of Markov chains, especially those arising in biochemical reaction networks.
Our findings  offer insights into how the structure of chemical reaction networks contributes to this phenomenon, which is essential for designing more efficient systems and avoiding potential misinterpretations of long-term system behavior.

Common general methods in proving lower bounds for mixing time include the spectral and geometric (conductance) methods, which involves
finding a  cut set (a set whose removal divides the state space into two disjoint subsets) with small conductance, bounding the Cheeger's constant or establishing a log-Sobolev inequality. They also include the coupling techniques, and a general method of finding
a distinguished statistic (a real-valued function $f$)  on the state space of the Markov chain $X$ such that a distance between the distribution of
$f(X(t))$ and the distribution of $f$ under the stationary distribution can be bounded from below. See for instance \cite{levin2017markov, montenegro2006mathematical, berestycki2014lectures} for the basic ideas of these general methods. 
Furthermore, the connection between the mixing time and moments of the first-passage time, or the hitting time of the Markov process has been used to study the mixing time; see \cite{veretennikov1997polynomial, peres2015mixing, bradley2005basic}. This approach is reminiscent to choosing the distinguished statistic $f$ to be an indicator of a set for the first-passage time. The method of lifting a Markov chain also bounds the mixing times from below by other mixing times \cite{chen1999lifting, ramanan2018bounds}.

Lower bounds of mixing times were obtained for many Markov chains including  birth-death chains \cite{chen2004markov, mitrophanov2007convergence}, random walks on finite graphs \cite{lyons2017probability}, the Ising model  \cite{ding2011mixing}, titling and shuffling \cite{wilson2004mixing},  and single-site dynamics on graphs  \cite{hayes2005general}. 
In \cite{hayes2005general}, the authors used a coupling method to study Markov models related to local, reversible updates on randomly chosen vertices of a bounded-degree graph, including Glauber dynamics for $q$-colorings on $n$-vertex graph with bounded degree. In \cite{ding2011mixing}, the authors used spectral gaps to quantify the lower bounds of  mixing times for reversible Glauber dynamics for the Ising model on a finite graph. Exponentially slow mixing behaviors for some finite state chains are also studied quite recently \cite{gheissari2018exponentially,tsunoda2021exponentially}.  
Lower Bounds for mixing Coefficients of diffusion processes were also studied; see for instance \cite{kabanov2006lower}. 
However, for non-reversible Markov chains on countably infinite state spaces where stochastic calculus is not easily applicable,  
a lower bound of mixing times like \eqref{Def:slow-mixing1} has not been much explored.

In this paper, we study the concept of slow mixing for general, possibly non-reversible continuous-time Markov processes whose state space are infinite subsets of the positive quadrant $\Z_{\ge 0}^d$, where $d$ is an arbitrary integer dimension. Such processes include models for  biochemical reaction networks that describe the stochastic dynamics of the copy numbers of $d$ chemical species.  
For biochemical systems, the boundary $\partial \mathbb Z_{\ge 0}^d$ corresponds to the states at which one or more chemical species go extinct. At those states, some reactions may be shut down or turned-off and thereby lead to slow mixing. Our focus here is to study the boundary behaviors of the corresponding Markov chains.

Our general result, Theorem \ref{thm}, 
not only quantifies ``slow-mixing behavior", but also provides general criteria for such global behavior in terms of the local behavior of the Markov chain near the boundary of the state space.
Here the terminology ``slow-mixing" means that there exist constants $\theta>1$ and $C_{\delta}>0$ 
such that  
\begin{equation}\label{Def:slow-mixing1}
t^{\delta}_{\rm mix}(x) \ge C_{\delta}\,|x|^\theta   
\end{equation}
for all initial state $x$
on a boundary face of the quadrant, where $\theta$ does not depend on the choice of $\delta$ nor on $x$.
The local behavior of the Markov chain near the boundary that guarantees such slow mixing is  intuitive: that the Markov chain  is trapped ``near" the boundary of the state space for ``long" enough, where the parameter $\theta$ encodes this local behavior.
The lower bound \eqref{Def:slow-mixing1}
implies that $X$ is non $L^2$-exponentially ergodic if the tail of $\pi$ does not decay fast enough, as we precised in 
Remark \ref{rmk:nonErgodic}. 

We are not aware of such lower bound for continuous-time Markov chains with countably infinite state space in any previous work. For finite state space, there is a previously considered concept of `rapid mixing' that indicates the mixing time is of order $n\log n$ for some parameter $n$ related to the size of the state space \cite{hayes2005general}.
For the Markov chains modeling the biochemical systems considered in this paper, partly due to the infinite state space, typical approaches for lower bounds of mixing times such as the spectral gap, the coupling and the conductance methods 
do not seem to be directly applicable. Our method of proof is inspired by choosing a distinguished statistic $f$ to be an indicator of a set which has asymptotically small hitting probability for ``large time", as the initial state $x$ also tends to infinity, and quantifies the relationship among three things: the magnitude of the  hitting probability, how large is the "large time", and how fast $|x|\to\infty$.

The main application of our general result is for Markov chains modeling biochemical reaction networks such as  the following example: 
\begin{align}\label{model2}
    B \xrightleftharpoons[1]{1} 2B, \quad \emptyset \xrightleftharpoons[1]{1} A+B.
\end{align}
which describes the interaction of two chemical species ($A$ and $B$) in reaction systems.
For example, when a Poisson process assigned to the reaction $A+B\to \emptyset$ in \eqref{model2} jumps up at a random time, then one copy of species $A$ and one copy of species $B$ annihilate each other.
Under a typical choices of kinetics, the so-called `mass-action kinetics', the transition rates of the Markov chains associated with the reaction network are polynomials of the current state. 
In the literature of stochastically modeled reaction systems, one of the main interests has been identifying classes of reaction systems that admit some dynamical behaviors such as ergodicity \cite{anderson2018some, anderson2020tier, anderson2020stochastically, anderson2010product, xu2023full}, explosion \cite{anderson2018non}, exponential ergodicity \cite{anderson2023new, fan2023constrained, mixing_And_Kim, xu2023full}, non-exponential ergodicity \cite{kim2024path} and upper bounds of mixing times \cite{mixing_And_Kim, anderson2023new}. However, to our best knowledge, there was no general result that guarantees lower bounds of mixing times  for
reaction networks as considered in this work. 

Our general criteria in Theorem \ref{thm} identify a class of biochemical reaction networks whose associated Markov chains exhibit slow mixing. These networks have two features: they consist of two species and the reactions form a cycle  as in \eqref{eq:cyclicmodel}, generating paths on the boundary of the state space that satisfy our slow mixing criteria. 
This application reveals an interesting dynamical feature of stochastic reaction networks, which we called ``boundary-induced slow mixing", is not systematically studied before.
Although how the boundary of the state space slows down the mixing time is not quantified before,
the boundary is known to typically induce interesting dynamical features such as the discrepancy between deterministic and stochastic models \cite{anderson2019discrepancies, bibbona2020stationary}.
Our application to 2-dimensions stochastic reaction networks also contrasts the results in \cite{xu2023full}, where $1-$dimensional continuous-time Markov chains are classified by their long-term behaviors such as ergodicity, exponential ergodicity, and quasi-stationary distributions. This is a motivation for us to focus on $2$-dimensional Markov chains where much less is known.

Our results  shed lights on how to increase the mixing rate of a stochastic reaction network in practice, by trimming the gazillion possibilities about how we may modify the network so that it no longer satisfies the assumptions in Theorem \ref{thm}. For example, the stochastic reaction network \eqref{model2} satisfies  the assumptions in Theorem \ref{thm} and therefore exhibits slow mixing; see example  \ref{ex: toy example} for details. However, if we add in-flow and out-flow with rate 1 for both species in \eqref{model2}, then the continuous Markov chain associated with the new model has the same stationary distribution as before, but now the mixing time is much shorter. The mixing time is now bounded above by $O(|x|\log |x|)$ as $|x|\to\infty$ \cite{mixing_And_Kim}.
In Section \ref{sec: upper bound}, we obtain also an upper bound for a first passage time of the reaction networks given in \eqref{eq:cyclicmodel} which suggests that our lower bound for \eqref{eq:cyclicmodel} can be sharp, as illustrated by our simulations.
This upper bound is of independent interest and extends exiting results for passage time moments for reflected random walks in the non-negative quadrant \cite{aspandiiarov1996passage, menshikov1996passage}.

This paper is organized as follows. In Section \ref{sec:markov}, we briefly recall some basic definitions about continuous-time Markov chains and give the key conditions for our main result, Theorem \ref{thm}. In Section \ref{sec: stochastic crn}, we introduce reaction networks as our main application, and provide the class of reaction networks that admit slow mixing. The proofs of main theorems and lemmas in Sections \ref{sec:markov} and \ref{sec: stochastic crn} are given in Sections \ref{sec:proof 1}--\ref{sec:proof 3}. A table of notations and features of stationary distributions of our main application models are given in Appendix \ref{app:table} and \ref{app:stationary}, respectively.

\bigskip




\section{A general result for continuous-time Markov chains in $\Z_{\ge 0}^d$}\label{sec:markov}

Let $X=\big(X(t)\big)_{t\in \R_{\ge 0}}$ be a \textbf{continuous-time Markov process}  in the non-negative quadrant $\Z_{\ge 0}^d$ of dimension  $d\ge 2$. Roughly, for any pair of distinct states $x$ and $z$ in $\Z_{\ge 0}^d$,
\begin{equation}\label{qxz}
    \mathbb{P}(X(t+\Delta t)=z \ | \ X(t)=x)=q_{x,z}\,\Delta t + o(\Delta t),  \quad \text{as}\quad  \Delta t \to 0.
\end{equation}
where $q_{x,z}\in\R_{\ge 0}$ is a constant 
called  the transition rate  from  $x$ to $z$. 
See, for instance,
\cite{NorrisMC97} for the 
basic theory of continuous-time Markov processes. 
Let
$X^D=(X^D(k))_{k\in \Z_{\ge 0}}$ be the \textbf{embedded discrete-time Markov chain} of $X$. That is, $X^D(k):=X(T_k)$  where $T_k$ is the $k$ th jump time of $X$ for $k=1,2,3,\ldots$ and $T_0=0$.  The one-step transition probabilities of $X^D$ satisfy
\begin{align}\label{eq:transition prob of XD}
\P(X^D(k+1)=z | X^D(k)=x)=\frac{q_{x,z}}{\sum_{w\in \Z_{\ge 0}^d\setminus \{x\}} q_{x,w}}
\end{align}
for any pair of distinct states $x$ and $z$, whenever  $\displaystyle \sum_{w\in \Z_{\ge 0}^d\setminus \{x\}} q_{x,w}>0$. 


For a sequence  $\eta=(\eta_i)_{i=1}^L$ of elements in $\Z^d$, we let $|\eta|:=L$ be its length and  $E_{\eta}$ be the event that the first $|\eta|$ steps of  $X^D$ follow $\eta$. That is,
        \begin{align}\label{eq:dominantcycle}
           E_{\eta}:=\{X^D(i)=X^D(i-1)+\eta_i \quad \text{for } i=1,2,\dots,|\eta|\}.
        \end{align}


Let $\P_x$ be the probability measure under which $X(0)=x$ almost surely (i.e. starting at $x$). Under $\P_{x}$,  $E_{\eta}$  is the event that the trajectory of $X^D$ during the first $|\eta|$ steps is the directed path in $\Z^d$ that starts at $x\in\Z^d$ and have increments  $(\eta_i)_{i=1}^{|\eta|}$.
Let
\begin{equation}\label{Def:gammaeta}
\gamma_{\eta}:=\left({\bf 0},\, \eta_1,\, \eta_1+\eta_2, \,\ldots,\,\sum_{i=1}^{|\eta|} \eta_i\right)
\end{equation}
be the directed path in $\Z^d$ that starts at ${\bf 0}\in\Z^d$ and have increments  $(\eta_i)_{i=1}^{|\eta|}$. The path $\gamma_{\eta}$ is a cycle if and only if $\sum_{i=1}^{|\eta|}\eta_i={\bf 0}$. 

\begin{example}
For the sequence 
$$\eta=\left\{
\begin{pmatrix}
    1 \\ 1
    \end{pmatrix},\;
    \begin{pmatrix}
    -1 \\ -1
    \end{pmatrix}
\right\},$$
the path $\gamma_{\eta}=\big({\bf 0},\, (1,1),\, {\bf 0} \big)$ is a cycle. Under  $\P_{(n,0)}$,  the event $E_{\eta}$ defined in \eqref{eq:dominantcycle}  is the event that the trajectory of $X^{D}$ during the first 2 steps is the directed path $\big( (n,0),\, (n+1,1),\, (n,0) \big)$.   \hfill $\triangle$
\end{example}

We shall consider the boundary face $\{x_d=0\}\cap \Z_{\ge 0}^d$ of the non-negative quadrant and we let  $\|x\|_{\infty}=\max_{1\leq i\leq d}|x_i|$ where $x=(x_1,x_2,\dots,x_d) \in \mathbb Z^d_{\ge 0}$
Clearly, $\{x_d=0\}\cap \Z_{\ge 0}^d=\cup_{n=0}^{\infty} \mathbb{I}_n$  where
\begin{equation}\label{Def:Initialn}
\mathbb{I}_n:= \{ x\in \Z_{\ge 0}^d:\; x_d=0,\;\|x\|_{\infty}=n\}  
\end{equation}
is the set of points on the boundary face  that is of distance $n$ from the origin. 
For dimension $d=2$,  $\mathbb{I}_n=\{(n,0)\}$ is a single point. 

The following is our key assumption on  $X^D$. It implies, by the strong Markov property of $X^D$, that $X^D$ spends much time on the boundary face $\{x_d=0\}\cap \Z_{\ge 0}^d$   whenever the process hits this boundary at a location of distance at least $N_0$ from the origin.

\begin{assumption}\label{assump}
    Suppose there exist $N_0\in \mathbb{N}$ and $\theta_1 \in (0,\infty),\,\theta_2\in [\theta_1,\infty)$ such that the followings hold for any integer $n\geq  N_0$. 
    \begin{enumerate}
        \item[(i).] (Dominating cycles) There exists a subset $\mathcal{T}^{(1)}$ of  sequences in $\Z^d$ with finite lengths  and a constant $c_1\in (0,\infty)$ such that $\gamma_{\eta}$ is a cycle (i.e. $\sum_{i=1}^{|\eta|}\eta_i={\bf 0}$) for all $\eta\in \mathcal{T}^{(1)}$  and 
        \begin{align}\label{eq:assump_1st}
           \inf_{x\in \mathbb{I}_n}\P_{x}\left( \bigcup_{\eta \in \mathcal{T}^{(1)}} E_{\eta} \right ) \geq &\, 1 - \frac{c_1}{n^{\theta_1}}. 
        \end{align}     
        \item[(ii).] (Dominating excursions)  There exists a finite (possibly empty) subset $\mathcal{T}^{(2)}$ of sequences in $\Z^d$ with finite lengths and a constant $c_2\in (0,\infty)$ such that  the path $\gamma_{\eta}$ ends on the set $\{x_d=0\} \setminus \{{\bf 0}\}$
        for all $\eta\in \mathcal{T}^{(2)}$  and
        \begin{align}\label{eq:assump_3rd}
                \inf_{x\in \mathbb{I}_n}\P_{x}\left( \bigcup_{\eta \in \mathcal{T}^{(1)} \cup \mathcal{T}^{(2)}} E_{\eta} \right ) \geq &\, 1 - \frac{c_2}{n^{\theta_2}}.
        \end{align}
    \end{enumerate}
\end{assumption}

Note that $\gamma_{\eta}$ is an ``excursion" from the boundary face $\{x_d=0\}\cap \Z_{\ge 0}^d$  for each $\eta\in \mathcal{T}^{(1)}\cup\mathcal{T}^{(2)}$ in the sense that the path $\gamma_{\eta}$ starts and ends on the boundary face. Furthermore,  those paths corresponding to $\mathcal{T}^{(1)}$ are cycles and those corresponding to $\mathcal{T}^{(2)}$ are \textit{not} cycles. In particular, $\mathcal{T}^{(1)}$ and $\mathcal{T}^{(2)}$ are disjoint.



\begin{remark}\label{Rk:T1}
Assumption \eqref{eq:assump_1st} asserts that, starting at  $x\in \mathbb{I}_n$ for large $n$, the process $X^D$ will stay in one of the cycles in $\mathcal{T}^{(1)}$ for a long time. Precisely, let $R$ be the number of returns to $x\in \mathbb{I}_n$ before exiting a cycle in $\mathcal{T}^{(1)}$. By \eqref{eq:assump_1st} and the strong Markov property, 
$$\inf_{x\in \mathbb{I}_n}\P_{x}\left( R\geq k \right ) \geq   \left(1-\frac{c_1}{n^{\theta_1}}\right)^k$$ for all $k\geq 1$. In particular, the expected number of returns $\E_{x}[R]\geq \dfrac{n^{\theta_1}}{c_1}$ for all $x\in \mathbb{I}_n$. 
\end{remark}

\begin{remark}\rm\label{Rk:T2a}
We allow $\mathcal{T}^{(2)}$ to be an empty set. In this case,  our main results (Theorem \ref{thm} and Corollaries \ref{cor:positiverecurrence} and \ref{cor:mixing}) still hold with $\theta_2=\theta_1$; see Corollary \ref{Cor:Thm}. These results are weaker in general and are much simpler to prove, but already offer potential applications. For instance, as illustrated in Theorem \ref{prop:cycle}  in Section \ref{sec:class of reaction networks} for stochastic reaction networks, if one can find a cycle on a boundary of $\Z_{\ge 0}^d$ on which the process will follow with high probability, then one may apply Corollary \ref{Cor:Thm}. 
\end{remark}

\begin{remark}\rm\label{Rk:T2b}
If $\mathcal{T}^{(2)}$ is non-empty and $\theta_2>\theta_1$, 
then assumption \eqref{eq:assump_3rd} will lead to a stronger lower bound in Theorem \ref{thm} and Corollaries \ref{cor:positiverecurrence} and \ref{cor:mixing}. This 
stronger lower bound 
 appears to be sharp for the class of examples considered in Section \ref{sec:class of reaction networks}, as our upper bound \eqref{eq:mean first passage time} and our simulation studies  in Section \ref{S:simulations} suggest.

\end{remark}

We introduce the stopping times $\{\nu_i\}_{i=0}^{\infty}$ and $\{\mu_i\}_{i=1}^{\infty}$  that are respectively the  $i$ th visit to and the $i$ th exit from the boundary face $\{x_d=0\}\cap \Z_{\ge 0}^d$. Specifically, we let $\nu_0=0$ and for $i\geq 1$, 
\begin{align}\label{eq:StoppingDef}
    \mu_i := \inf \{t > \nu_{i-1}:X_d(t) > 0 \} \quad \text{and}\quad  \nu_i := \inf \{t > \mu_i: X_d(t) = 0\},
\end{align}
with the convention that $\inf\emptyset=+\infty$, where  $X(t)=(X_1(t),X_2(t),\ldots,X_d(t))$.
We shall assume that $\{\nu_i\}_{i\geq 1}$ are all finite almost surely, so that with full probability, 
$$0=\nu_0< \mu_1<\nu_1<\mu_2<\nu_2<\cdots<\nu_{i-1}<\mu_i<\nu_i< \cdots$$



\begin{assumption}[Return time to boundary]\label{assump2}
For any starting point on the boundary face $\{x_d=0\}\cap\Z_{\ge 0}^d$, 
the return times $\{\nu_i\}_{i\geq 1}$ are finite almost surely and stochastically bounded below by an exponential random variable with intensity $\kappa$, where $\kappa\in(0,\infty)$ is a constant. That is, for all  $x\in \{x_d=0\}\cap\Z_{\ge 0}^d$, we have $\P_x(\nu_i<\infty)=1$ for all $i\in\Z_{\ge 0}$ and $\P_x(\nu_1 \ge t)\ge  e^{- \kappa\,t}$  for all $t\in\R_{\ge 0}$. 
\end{assumption}

Our general result is Theorem \ref{thm} below, which works for any dimension $d\geq 2$. This theorem  implies a lower bound for the total variation distance between $X$ and its stationary distribution, and for  the mixing time, when $X$ admits a unique stationary distribution. The latter are precised in Corollary \ref{cor:positiverecurrence} and Corollary \ref{cor:mixing} respectively.
For $x=(x_1,x_2,\ldots,x_d)\in \Z^d$ we write $\|x\|_{d-1,\infty}:=\max_{1\leq i\leq d-1}|x_i|$.

\begin{thm}\label{thm}
Let $X=\big(X(t)\big)_{t\in \R_{\ge 0}}$ be a continuous-time Markov process on $\Z_{\ge 0}^d$ that satisfies  Assumption \ref{assump}  and Assumption \ref{assump2}. Let   
\begin{equation}\label{Def:theta}
     \theta:=(1+\theta_1)\wedge \theta_2 := \min\{1+\theta_1,\,\theta_2\}.
\end{equation} 
Then for any $\delta\in(0,1)$, there exist  constants $N_{\delta}\in \mathbb{N}$ and $C_{\delta}\in (0,\infty)$ such that for any $n> N_\delta$ and any initial state $x\in \mathbb{I}_n$,
    \begin{equation}\label{eq:problowerbdd}
   \inf_{t\in[0,\,C_{\delta} n^{\theta}] } \;\P_{x} \left( \|X(t)\|_{d-1,\infty}>\frac{n}{2} \right) \geq \delta,
    \end{equation} 
where $\|X(t)\|_{d-1,\infty}=\max_{1\leq i\leq d-1}|X_i(t)|$.
\end{thm}
Note that $\delta$ can be arbitrarily close to $1$ in Theorem 
\ref{thm}. So the lower bound \eqref{eq:problowerbdd} is the best possible under our assumptions. 

\medskip

Next, we make a further assumption on $X$. This assumption is satisfied, for instance, when $X$ admits a unique stationary distribution
(see \cite{NorrisMC97} for the precise definition). 
\begin{assumption}\label{A:Positiveecurrent}
For all $n\in \mathbb{N}$ large enough and $x\in  \mathbb{I}_n$, there exists a unique stationary distribution $\pi_x$ on the communication class of the Markov process $X$ containing the state $x$. Furthermore, 
    \begin{align*}
         \lim_{n\to \infty} \sup_{x\in \mathbb{I}_n} \pi_x(\Lambda_n)=0,
    \end{align*}
where
\begin{equation}\label{Def:Lambdan}
\Lambda_n := \left\{x \in\Z_{\ge 0}^d:\, \|x\|_{d-1,\infty} > \frac{n}{2}\right\}.
\end{equation}
\end{assumption}
Clearly, if $X$ admits a unique stationary distribution $\pi$, then Assumption \ref{A:Positiveecurrent} is satisfied with $\pi_x=\pi$ for all $x\in \Z_{\ge 0}^d$. For ease of reading, the reader may first focus on the latter case.

\begin{corollary}\label{cor:positiverecurrence}   
Suppose, in addition to all assumptions in Theorem \ref{thm}, that Assumption \ref{A:Positiveecurrent} holds. Then for any $\delta\in (0,1)$, there exist constants $N'_\delta\in \mathbb{N}$ and $C'_{\delta}\in (0,\infty)$, such that for any $n> N'_\delta$ and any initial state $x\in \mathbb{I}_n$, 
\begin{equation}\label{eq:TV_bound}
\inf_{t\in[0,\,C'_{\delta}\, n^{\theta}] }\Vert P^{t}(x, \,\cdot\,) - \pi_x \Vert_{\rm TV} \geq  \delta,
\end{equation} 
where $P^{t}(x,\,\cdot \,)$ is the probability distribution of $X(t)$ under $\P_x$, and $\theta:= \min\{1+\theta_1,\,\theta_2\}$.
\end{corollary}

\begin{proof}[Proof of Corollary \ref{cor:positiverecurrence}] 

Fix $\delta \in (0, 1)$. 
By Theorem \ref{thm}, there exist constants $N'_\delta\in \mathbb{N}$ and $C'_{\delta}\in (0,\infty)$, such that for all $n> N'_\delta$ and for any  initial state $x\in \mathbb{I}_n$,
\begin{equation}
    \inf_{t\in[0,\,C'_{\delta}\, n^{\theta}] }  P^{t}(x,\Lambda_n)=  \inf_{t\in[0,\,C'_{\delta}\, n^{\theta}] } \P_{x} \left( \|X(t)\|_{d-1,\infty}> \frac{n}{2} \right)  \geq  \delta+\frac{1-\delta}{2}.
\end{equation} 
On the other hand, by Assumption \ref{A:Positiveecurrent}, 
for any  $\epsilon\in(0,\,\frac{1-\delta}{2})$ there exists $N_{1,\epsilon}$ such that
\[
  \sup_{n\geq N_{1,\epsilon}} \,\sup_{x\in \mathbb{I}_n} \pi_x(\Lambda_n)  < \epsilon. 
\]
Hence for all $n > \max\{N_{1,\epsilon}, N'_\delta\}$ and any initial state $x\in \mathbb{I}_n$,
\begin{equation}
  \inf_{t\in[0,\,C'_{\delta}\, n^{\theta}] }\Vert P^{t}(x,\,\cdot\,) - \pi_x \Vert_{\rm TV} \;\geq\; \inf_{t\in[0,\,C'_{\delta}\, n^{\theta}] }  P^{t}(x,\Lambda_n)   -   \pi_x(\Lambda_n)  \geq  \delta+\frac{1-\delta}{2}-\epsilon\geq \delta.
\end{equation}
\end{proof}

An immediate consequence of Corollary \ref{cor:positiverecurrence} is that the mixing time of $X$ is at least $O(n^{\theta})$ when the process starts at  $x\in \mathbb{I}_n$, as $n\to\infty$. For any $\delta\in (0,\infty)$ and any
$x\in\Z_{\ge 0}^d$, 
the \textbf{mixing time of  $X$ starting at $x$} with threshold $\delta$ is defined as
 \begin{align*}\label{eq:mixing time}
t^{\delta}_{\rm mix}(x) := \inf \{t \ge 0 : \Vert P^t(x,\,\cdot\,) - \pi_x \Vert_{\rm TV} \le \delta\}.
\end{align*}
\begin{corollary}\label{cor:mixing}
Under the same assumptions of Corollary \ref{cor:positiverecurrence},
for any $\delta\in (0,\,1)$ we have 
\begin{equation}
\inf_{x\in \mathbb{I}_n}t^{\delta}_{\rm mix}(x) \geq C'_{\delta}\, n^{\theta} \qquad \text{for all }n\geq N'_\delta,
\end{equation}
where the constants $C'_{\delta}$ and $N'_\delta\in(0,\infty)$  are the same as those in Corollary \ref{cor:positiverecurrence}.
\end{corollary}

\begin{remark}[Non $L^2$-exponentially ergodic]\rm \label{rmk:nonErgodic}
The lower bound in Corollary \ref{cor:mixing} typically implies that $X$ is non $L^2$-exponentially ergodic, if the tail of $\pi$ does not decay fast enough in the sense of \eqref{E:nonexpo} below.  
We now make this precise.

Suppose $X$ is an irreducible and positive-recurrent continuous-time Markov chain with countable state space $\S\subset \Z_{\ge 0}^d$ and stationary distribution $\pi$. We say that
$X$ is
\textbf{$L^2$-exponentially ergodic} if 
there is a constant $\eta \in(0,\infty)$ such that
\begin{equation}\label{ExpoEr23}
\left \|P_tf-\pi(f) \right \|_{L^2(\pi)} \leq e^{-\eta t}\, \left \|f-\pi(f) \right \|_{L^2(\pi)}
\end{equation}
for all compactly supported functions $f$ on $\S$ and $t\in\R_{\ge 0}$. 

Inequality \eqref{ExpoEr23} implies (for instance \cite[Lemma 2.4]{anderson2023new}) that for all $x\in \mathbb S$, 
\begin{align}\label{eq:ExpoB}
\Vert P^t(x,\,\cdot\,) - \pi \Vert_{\TV} \le \frac{2}{\pi(x)} e^{-\eta t} \quad \text{for } t \in\R_{\ge 0]}
\end{align}
and therefore also that
\begin{equation}\label{E:mixing_upper}
t^{\delta}_{\rm mix}(x) \leq \frac{1}{\eta}\ln\left(\frac{2}{\delta\,\pi(x)}\vee 2\right)
\end{equation}
for all $x$ and $\delta\in (0,1)$; see \cite[Theorem 2.5(ii)]{anderson2023new} for details.
Inequality \eqref{E:mixing_upper}  will contradict \eqref{E:mixing} if the tail of $\pi$ does not decay fast enough
in the following sense: if there exists  a sequence $\{x^{(n)}\}$ with $x^{(n)}\in \mathbb{I}_n$ such that 
$ \limsup_{n\to\infty}\pi(x^{(n)})\, e^{\eta\,n^{\theta}} \in(1,\infty]$.

Therefore, $X$ is {\it not} $L^2$-exponentially ergodicity (for any choice of $\eta$ in \eqref{ExpoEr23}) if $X$ satisfies the  assumptions of Corollary \ref{cor:positiverecurrence} with $\theta>0$ and if there exist $\theta'\in(0,\theta)$ and  a sequence $\{x^{(n)}\}$ with $x^{(n)}\in \mathbb{I}_n$  such that
\begin{equation}\label{E:nonexpo}
   \limsup_{n\to\infty}\pi(x^{(n)})\, e^{n^{\theta'}} \in (0,\infty].   
\end{equation}
For instance,  $X$ is {\it not} $L^2$-exponentially ergodicity if it 
satisfies the assumptions of Corollary \ref{cor:positiverecurrence} with
$\theta>1$ and if
$\pi$ is a Poisson distribution described in \eqref{E:pi_complexBalanced}. 
This is because then Condition \eqref{E:nonexpo} holds with $\theta'\in (1,\theta)$ and $x^{(n)}=(n,0,0,\ldots,0)$.
\end{remark}

As mentioned in Remark \ref{Rk:T2a}, a possibly weaker result  (where $\theta_2=\theta_1$) follows immediately if we omit part (ii) of Assumption \eqref{assump}.

\begin{corollary}\label{Cor:Thm}
Let $X=\big(X(t)\big)_{t\in \R_{\ge 0}}$ be a continuous-time Markov process on $\Z_{\ge 0}^d$ that satisfies 
Assumption \eqref{assump}(i) and Assumption \ref{assump2}. Then for any $\delta\in(0,1)$, there exist  constants $N_{\delta}\in \mathbb{N}$ and $C_{\delta}\in (0,\infty)$ such that 
for any $n> N_\delta$ and any initial state $x\in \mathbb{I}_n$,
\begin{equation}\label{eq:problowerbdd2}
   \inf_{t\in[0,\,C_{\delta}\, n^{\theta_1}] } \P_{x} \left( \|X(t)\|_{d-1,\infty}> \frac{n}{2} \right) \geq \delta.
    \end{equation} 
Suppose, furthermore, that Assumption \ref{A:Positiveecurrent} holds. Then  there exist constants $N'_\delta\in \mathbb{N}$ and $C'_{\delta}\in (0,\infty)$, such that for any $n> N'_\delta$ and any initial state $x\in \mathbb{I}_n$, 
\begin{equation}\label{eq:TV_bound_2}
\inf_{t\in[0,\,C'_{\delta}\, n^{\theta_1}] }\Vert P^{t}(x, \,\cdot\,) - \pi_x \Vert_{\rm TV} \geq  \delta \qquad \text{and}\qquad t^{\delta}_{\rm mix}(x) \geq C'_{\delta}\, n^{\theta_1}.
\end{equation} 
\end{corollary}

\begin{remark}[Boundary-induced slow-mixing behavior]\rm\label{rmk:slowmixing}
For a continuous-time Markov chain $X$ that satisfies all assumptions in Corollary \ref{Cor:Thm},  the trajectory is said to exhibit a ``boundary-induced slow-mixing behavior''. The reason is that 
the lower bounds of the mixing time described in Corollaries \ref{cor:mixing} and \ref{Cor:Thm} are roughly due to $X^D$ spending much time near the boundary face (Remark \ref{Rk:T1}). 
By the strong Markov property, whenever the process hits the boundary face $\{x_d=0\}\cap \Z_{\ge 0}^d$  at a location of distance $N_0$ or more from the origin, $X^D$ will spend much time near the boundary face, getting stuck in cycles in $\mathcal{T}^{(1)}$.
The stronger bound in Corollary \ref{cor:mixing} is due to further control on the rare excursions, described by $\mathcal{T}^{(2)}$,  when the trajectory  exits the cycles in $\mathcal{T}^{(1)}$.  This stronger bound appears to be sharp for the class of examples considered in Section \ref{sec:class of reaction networks}, as our upper bound \eqref{eq:mean first passage time} and our simulation studies suggest.
\end{remark}

\medskip


To explain how the assumption \eqref{eq:assump_3rd} with $\theta_2>\theta_1$ can lead to a stronger lower bound, it suffices to explain the rationale behind the exponent $\theta:= \min\{1+\theta_1,\,\theta_2\}$  in Theorem \ref{thm}. 
In Section \ref{S:simulations} we demonstrate some concrete examples. 

\begin{remark}[Why $\theta:= \min\{1+\theta_1,\,\theta_2\}$?] \rm\label{ex:motivating}
From the technical standpoint, this requirement of $\theta$ is used crucially in the proof of Proposition \ref{prop:cycle gives tauZ}, where  we will use this requirement  to obtain
the inequality \eqref{eq:problowerbdd_2} (equivalently \eqref{eq:thm1_also}) by establishing upper bounds for both
\begin{align}\label{eq:thm1_also2}
  \sup_{t\in[0,\,Cn^{1+\theta_1}]}  \P_{x}\left(\|X(t)\|_{d-1,\infty}\leq \frac{n}{2}  \right) \quad\text{and}\quad  \sup_{t\in[0,\,Cn^{\theta_2}]}  \P_{x}\left(\|X(t)\|_{d-1,\infty}\leq \frac{n}{2}  \right).
\end{align}

An intuitive explanation to this is as follows.
First,  Assumption \eqref{eq:assump_1st} implies that the expected number of returns to a starting point  $x\in \mathbb{I}_n$ before exiting a cycle in $\mathcal{T}^{(1)}$ is at least $O(n^{\theta_1})$; see Remark \ref{Rk:T1}. This implies, by
assumption \eqref{eq:assump_3rd}, that the expected number of return to the boundary face before exiting 
the set  $\left\{ \|x\|_{d-1,\infty}>\frac{n}{2} \right\}$
is \emph{at least}  $O(n^{\theta_1+1})$. This is because
when the trajectory of $X$ leaves $\mathcal{T}^{(1)}$, it will follow an excursion in $\mathcal{T}^{(2)}$ with high probability and return to the  boundary face  at a location ``close" to its starting point $x\in \mathbb{I}_n$. 
``close" here means within a fixed distance (independent of $n$) to the starting point  $x\in \mathbb{I}_n$, which in turn is due to the assumption that $\mathcal{T}^{(2)}$ is a finite subset  of sequences in $\Z^d$ with finite lengths.  Hence, by Assumption \ref{assump2},  the expected time to exit
the set $\left\{ \|x\|_{d-1,\infty}>\frac{n}{2} \right\}$ is \emph{at least}  $O(n^{\theta_1+1})$. This will lead to the first inequality in \eqref{eq:thm1_also2}.

Second, by a similar argument in  Remark \ref{Rk:T1}, 
Assumption \eqref{eq:assump_3rd} implies that
the expected number of returns to the subset $\left\{ \|x\|_{d-1,\infty}>\frac{n}{2} \right\} \cap \{x_d=0\}$ of the  boundary face before 
the trajectory of $X$ leaves $\mathcal{T}^{(1)}\cup \mathcal{T}^{(2)}$ is 
\emph{at least}  $O(n^{\theta_2})$. Hence the expected time to exit
the set $\left\{ \|x\|_{d-1,\infty}>\frac{n}{2} \right\}$ is \emph{at least}  
$O(n^{\theta_2})$ by Assumption \ref{assump2}. This will lead to the second inequality in \eqref{eq:thm1_also2}.
\end{remark}


\section{Applications to stochastic reaction networks}\label{sec: stochastic crn}

In this section we demonstrate the applicability of our results to a wide class of continuous-time Markov processes used heavily in biochemistry, ecology, and epidemiology. The models are referred to as \textbf{stochastic reaction networks}, for which continuous-time Markov processes are used to model the copy numbers of species as described below. We shall identify a class of  stochastic reaction networks with two species 
that 
have mixing times at least $O(n^{\theta})$ for some $\theta>1$ 
when the initial state is $(n,0)$, as described in Corollary \ref{cor:mixing}. 


Roughly, a reaction network is
a graphical configuration of interactions between chemical species such as oxygen and carbon dioxide in the Earth's atmosphere, or glucose and various amino acids in a biological cell. An example of reaction network with two species $\{X_1, X_2\}$ is
\begin{align}\label{ex:crn}
    X_1+X_2 \to 2X_1 \to 2X_2, \quad X_2\to \emptyset. 
\end{align}
The reactions $X_1+X_2 \to 2X_1  \to 2X_2$ describe that species $X_1$ and $X_2$ are combined to produce two copies of $X_1$ that can further be combined to  produce two copies of $X_2$, and the reaction $X_2\to \emptyset$ means that $X_2$ is washed out from the system. A general definition of reaction network is as follows.




\begin{definition}\label{def:21}
\emph{A  reaction network} is a triple of finite sets $(\mathcal S,\mathcal C,\mathcal R)$ where:
\begin{enumerate}
\item \emph{The species set} $\Sp=\{X_1,X_2,\cdots,X_d\}$ contains the species of the reaction network.
\item \emph{The reaction set} $\Re=\{R_1,R_2,\cdots,R_r\}$ is a collection of ordered pairs $(y,y') \in \Re$, with $y\ne y'$, where 
\begin{align}\label{complex}
y=\sum_{i=1}^d y_iX_i \hspace{0.4cm} \textrm{and} \hspace{0.4cm}
y'=\sum_{i=1}^d y'_iX_i
\end{align}
are non-negative linear combinations of the species, called complexes
and where the values $y_i,y'_i \in \mathbb{Z}_{\ge 0}$ are called the \emph{stoichiometric coefficients}. A pair $(y,y')\in \Re$ is called \emph{reaction} and is represented as $y\rightarrow y'$. We often represent a complex $y$ using a vector $(y_1,y_2,\dots,y_d)\in \mathbb Z^d_{\ge 0}$ as well as the summation shown in \eqref{complex}. Then the \emph{reaction vector} $y'-y$ for a reaction $y\to y'$ can mean `net change' via the reaction.
\item \emph{The complex set} $\C$  a non-empty set of non-negative linear combinations of the species in (\ref{complex}). Specifically, 
$\C = \{y\ : \ y\rightarrow y' \in \Re\} \cup \{y' \ : \ y\rightarrow y' \in \Re\}$.
\hfill $\triangle$
\end{enumerate}
\end{definition}

For the reaction network in \eqref{ex:crn}, we have $\S=\{X_1,X_2\}, \C=\{X_1+X_2, 2X_1, X_2, \emptyset\}$, and $\Re=\{ X_1+X_2 \to 2X_1, 2X_1 \to 2X_2, X_2\to \emptyset\}$. Note that $y=\emptyset$ is a complex such that $y_i=0$ for each 
$i$. For the reaction $X_1+X_2\to 2X_1$ in \eqref{ex:crn}, we present the complexes $y=X_1+X_2$ and $y'=2X_1$ by $(1,1)$ and $(2,0)$. Then $y'-y=(1,-1)$ means that we eventually gain one $X_1$ and lose one $X_2$.

The concentration (or the copy numbers) of species change over time according to their interactions specified by the reaction network. One of the main goals in reaction network theory is to study the relations between the topological structure of the reaction network and the associated dynamics of the species.

\medskip
\noindent
{\bf Deterministic model for chemical concentrations. }
The time evolution of the concentration of species can be modeled with ordinary differential equations, assuming the system is spatially well-mixed and the abundance of species in the system is sufficiently large.

Let $x(t)=(x_1(t),x_2(t),\dots,x_d(t)) \in \mathbb R^d_{\ge 0}$ where $x_i$ represents the concentration of the $i$ th species. Such an equation is
\begin{align}\label{eq:det_crn}
    \frac{d}{dt}x(t)=\sum_{y\to y' \in \mathcal R} \lambda_{y\to y'}(x(t))(y'-y).
\end{align}
Here $\lambda_{y\to y'}$ gives the intensity of a reaction $y\to y'$, which is a non-negative function on $\mathbb R^d_{\ge 0}$. A typically used intensity is 
\begin{align}\label{eq:det_mass}
    \lambda_{y\to y'}(x)=\kappa_{y\to y'}\prod_{i=1}^d x_i^{y_i} \text{ for some constant $\kappa_{y\to y'}>0$},
\end{align}
and the choice of this setting is called \emph{mass-action kinetics}. $\kappa_{y\to y'}$ is called a \textit{rate constant}. We typically incorporate the rate constants into the reaction network by placing them next to the reaction arrow as in $y\xrightarrow{\kappa_{y\to y'}} y'$. 


\medskip
\noindent
{\bf Stochastic model for copy numbers. }
When not all species in the system are abundant, it may be necessary to
 model the time evolution of the copy numbers of each species rather than their concentrations. 
Let $(\S,\C,\Re)$ be a reaction network with $|\S|=d$.
A reaction network equipped with a kinetics can be modeled as a continuous-time Markov process $X:=(X(t))_{t\geq 0}$, where $X(t)=(X_i(t))_{i=1}^d\in \Z_{\ge 0}^{d}$  and $X_i(t)$ is the copy number of the $i$ th species at time $t$. 

Let $q(x,x')$ denote the transition rate of $X$ from $x$ to $x'$. Under mild conditions such as conditions shown in \cite[Condition 1]{anderson2018non},
 the infinitesimal generator $\mathcal A$ of $X$ can be written as
 \begin{align}\label{InfGen}
\mathcal{A}f(x)= \sum_{x' \in  \Z^d_{\ge 0}} q(x,x')(f(x')-f(x))=\sum_{y \rightarrow y' \in \Re} \lambda_{y\rightarrow y'}(x)(f(x+y'-y)-f(x)),
\end{align}
for any function $f:\Z^d_{\ge 0} \to \R$ for which the right hand side of \eqref{InfGen} is defined. Alternatively, this means that 
\begin{align*}
    P(X(t+\Delta t)=z| X(t)=x)=\sum_{y'\to y \in \Re, z=x+y'-y}\lambda_{y\to y'}(x)\Delta t+o(\Delta t),   \quad \text{as}\quad  \Delta t \to 0.
\end{align*}
In this paper, we focus on the usual choice of intensity given by (stochastic) mass-action kinetics
\begin{align}\label{mass}
\lambda_{y\rightarrow y'}(x)= \kappa_{y\rightarrow y'} \prod_{i=1}^d \frac{x_i !}{(x_i-y_{i})!}\mathbf{1}_{\{x_i \ge y_{i}\}},
\end{align} 
where the positive constant $\kappa_{y\rightarrow y'}$ is the  reaction rate constant.

Note that
\begin{equation}
    \lambda_{y\rightarrow y'}(x)=0 \quad\text{if}\quad  x\notin y+\Z_{\ge 0}^d.
\end{equation}
That is, $\lambda_{y\rightarrow y'}$ vanishes on the set $\Z_{\ge 0}^d \setminus (y+\Z_{\ge 0}^d)$. In other words, if $y\neq 0$, then $\lambda_{y\to y'}(x)$ is zero when $x\in \partial \Z^d_{\ge 0}$.
For example, for the complex $y=X_1+X_2=(1,1)$ in \eqref{ex:crn}, then $\lambda_{y\to y'}(x)=0$ if $x\in \{(x_1,x_2): x_1\le 1\}\cup \{(x_1,x_2): x_1\le 1\}=\Z^2_{\ge 0}\setminus (y+\Z^2_{\ge 0})$. 
This `boundary effect' is one of the key ingredients for the conditions in Assumption \ref{assump} to hold.

Finally, as we highlight above, we study a lower bound of mixing times for a class of reaction networks. Hence it is also important to know which class of reaction networks admit a unique stationary distribution. There is a well-known theorem that shows that the associated mass-action Markov chains for any so-called `complex balanced reaction network' admits a unique stationary distribution \cite{anderson2010product}. A reaction network $(\S,\C,\Re)$ with given rate constants is \textbf{complex balanced} if there exists $c\in \R^d_{\ge 0}$ such that inflows and outflows at each complex $y$ are balanced at $c$. In other words, it holds that for each $y\in \C$
\begin{align}\label{E:complexBalance}
    \sum_{y\to y'\in \Re}\lambda_{y\to y'}(c)=\sum_{y''\to y\in \Re}\lambda_{y''\to y}(c),
\end{align}
where the intensities $\lambda_{y\to y'}$'s are given by \eqref{eq:det_mass}.
The constant vector $c\in \R^d_{\ge 0}$ is called a complex balanced steady state. Theorem \ref{thm:def0} establishes that, when stochastically modeled, complex balanced reaction networks possess stationary distributions that can be expressed in a product form of Poissons.

\begin{thm}[Anderson, Craciun and Kurtz 2010 \cite{anderson2010product}] \label{thm:def0} Under a choice of parameters $\kappa_i$'s, assume that a reaction network $(\S,\C,\Re)$ is complex balanced with a steady state $c =(c_1,\dots,c_d)\in \R^d_{>0}$. Then the associated Markov chain modeled under mass-action kinetics admits a unique (up to a communication class) stationary distribution $\pi$ such that \eqref{E:pi_complexBalanced} holds.
\begin{align}\label{E:pi_complexBalanced}
    \pi(x) \propto \prod_{i=1}^d\frac{c_i^{x_i}}{x_i!} \qquad \text{for  $x=(x_1,\dots,x_d)\in \Z_{\ge 0}^d$}.
\end{align}
\end{thm}

\subsection{A class of stochastic reaction networks with slow mixing}\label{sec:class of reaction networks}

In this section, we demonstrate the applicability of our general theorem, Theorem \ref{thm}, to a class of cyclic reaction networks with two species and an arbitrary number of complexes. We first describe these reaction networks
in \eqref{eq:cyclicmodel} below. We then show that the associated continuous-time Markov chain $X$ satisfies Assumptions  \ref{assump} and \ref{assump2}, provided that Assumption \ref{assump:cyclic} about the stoichiometric coefficients holds.


We consider an arbitrary    reaction network with two species ($A$ and $B$) of the following form:\vspace{-0.8cm}
\begin{equation}\label{eq:cyclicmodel}
   \begin{tikzpicture}[baseline={(current bounding box.center)}, state/.style={circle,inner sep=2pt}]
   \node[state] (1)  at (0,0) {$\alpha_1 A+\beta_1 B$};
  \node[state] (2)  at (3,0) {$\alpha_2 A+\beta_2 B$};
  \node[state] (dots)  at (5.5,0) {$ \ \cdots \ $};
  \node[state] (3)  at (8.5,0) {$\alpha_{L-1} A+\beta_{L-1} B$};
\node[state] (0) at (4.5,-1.5) {$\emptyset$};
    \path[->]
     (1) edge node[above] {$\kappa_1$} (2)
      (2) edge node[above] {$\kappa_2$} (dots)
       (dots) edge node[above] {$\kappa_{L-2}$} (3)
       (3) edge node[above] {$\kappa_{L-1}$} (0)
       (0) edge node[above] {$\kappa_0$} (1);     
   \end{tikzpicture}
\end{equation}
where $L\geq 2$ is an integer, $\{\alpha_i,\beta_i\}_{i=1}^{L-1}$ are non-negative integers and $\{\kappa_i\}_{i=0}^{L-1}$ are positive constants. That is, we let $\{\alpha_i,\beta_i\}_{i=1}^{L-1}$ be the stoichiometric coefficients and $\{\kappa_i\}_{i=0}^{L-1}$ be the rate constants. 
The number of complexes (including the empty complex $\emptyset$) is $L$ if the pairs $\{(\alpha_i,\beta_i)\}_{i=0}^{L-1}$ are distinct, where  $\alpha_0 =\beta_0 = 0$.

We denote the complexes  by $z_i = \alpha_i A + \beta_i B$ for $i=0,1,...,L-1$. Then under mass-action kinetics \eqref{mass}, the reaction rate function for $z_i \rightarrow z_{i+1}$ where $z_L := z(0) = \emptyset$. ,
\[
    \lambda_i (x) = \kappa_i \frac{x_A!}{(x_A-\alpha_i)!}\frac{x_B!}{(x_B-\beta_i)!}, \text{   for } i = 0,1,...,L-1.
\]


\begin{thm}\label{prop:cycle}
Consider a reaction network  given in \eqref{eq:cyclicmodel}, where $L\geq 2$ is an integer, $\{\alpha_i\}_{i=1}^{L-1}$ and $\{\beta_i\}_{i=1}^{L-1}$ are non-negative increasing integers, and $\{\kappa_i\}_{i=0}^{L-1}$ is a set of arbitrary positive constants. Let $X$ be the associated continuous-time Markov process under mass-action kinetics. Then Assumptions \ref{assump}(i) holds with $\theta_1 = \min_{i} \{\alpha_{i-1} - \alpha_{i-2}: 2\leq i \leq L\}$, and Assumption \ref{assump2} is satisfied with $\kappa=\kappa_0$. In particular, Theorem \ref{thm} holds with $$\theta= \theta_1=\min_{i} \{\alpha_{i-1} - \alpha_{i-2}: 2\leq i \leq L\}.$$ 
\end{thm}


Next, to strengthen Theorem \ref{prop:cycle} from $\theta=\theta_1$ to $\theta=1+\theta_1$, we need to impose further assumptions on the stoichiometric coefficients. Assumption \ref{assump:cyclic} below is one such example. 

\begin{assumption}\label{assump:cyclic}
The  constants $\{\alpha_i,\beta_i\}_{i=1}^{L-1}$ are such that
\begin{enumerate}
    \item $\{\alpha_i\}_{i=1,...,L-1}$ is a sequence of increasing integers such that $2\alpha_{i}-\alpha_{i+1}- \alpha_{i-1}\not = 0$ for all $i=1,2,...,L-2$ and $\alpha_{L-1}-\alpha_{L-2}-\alpha_1 \not = 0$,  and 
    \item $\beta_i = i$ for $i=0,1,..,L-1$.
\end{enumerate}
\end{assumption}

We present this assumption because it simplifies the identification of the trajectory  $\eta^i$ in $\mathcal{T}^{(2)}$:
upon exiting from the dominant cyclic trajectory $\eta^0$, the dominant reactions can be easily identified by the number of B species. In particular for any $1\leq j\leq L$, if there are $j$ copies of B species, then reaction $z_j \rightarrow z_{j+1}$ is the dominant reaction. The curious reader can see \eqref{eq:example network} below for an example  reaction network that falls into the class of reaction networks \eqref{eq:cyclicmodel} and satisfies Assumption \ref{assump:cyclic}. 

\begin{thm}\label{prop:cycle2}
Consider a reaction network  given in \eqref{eq:cyclicmodel} and assume all conditions of Theorem \ref{prop:cycle} hold.
Suppose, furthermore, that Assumption \ref{assump:cyclic} holds. 
 Then Assumptions \ref{assump}(i) and (ii) hold with 
\[
    \theta_1 = \min_{i} \{\alpha_{i-1} - \alpha_{i-2}: 2\leq i \leq L\}, \qquad  \theta_2  = \min_{2\leq i\leq L: \left(\alpha_{i-1} - \alpha_{i-2}\right)>\theta_1} \left(\alpha_{i-1} - \alpha_{i-2}\right) \wedge 2\theta_1 > \theta_1,
\]
and Assumption \ref{assump2} is satisfied with $\kappa=\kappa_0$. In particular, Theorem \ref{thm}, and Corollary \ref{Cor:Thm}  hold with 
\[
    \theta=1+\theta_1 =1+ \min_{i} \{\alpha_{i-1} - \alpha_{i-2}: 2\leq i \leq L\}.
\]
\end{thm}

\begin{remark}\label{rem:relax uniq for crns}
In Lemma \ref{lem:non unique}, we  show that for any stochastic reaction networks of the form \eqref{eq:cyclicmodel}, if a unique stationary distribution exists, then Assumption \ref{A:Positiveecurrent} is implied by condition 1 of Assumption \ref{assump:cyclic}. Hence the conclusions of Corollaries \ref{cor:mixing} and \ref{Cor:Thm} hold  under the conditions of Theorem \ref{prop:cycle2}.
\end{remark}


\begin{remark}
  The rate constants $\kappa_i$'s can be always selected so that the associated stochastic model $X(t)$ admits a stationary distribution. For example, if $\kappa_i$ is equal to a positive constant $c$ for all $i\in \{0,1,...,L-1\}$, then $(1,1)^\top$ is a complex balanced steady state \eqref{E:complexBalance}. This can be checked for each complex $\alpha_i A + \beta_i B$:
  \[
    cx_A^{\alpha_{i-1}} x_B^{\beta_{i-1}} =  c = cx_A^{\alpha_{i}} x_B^{\beta_{i}},
  \]
  and hence $X(t)$ has a unique stationary distribution  on the closed subspace accessible from the initial condition by Theorem \ref{thm:def0}. 
  Consequently, under such choice of rate constants $\kappa_i$'s, Corollary \ref{cor:mixing} holds with $\theta$ given in Theorem \ref{prop:cycle2}.
\end{remark}

\begin{remark}\label{rem:simple cyclic}
Using a simple case of the reaction networks \eqref{eq:cyclicmodel}, we visualize the paths in the subset $\mathcal T^{(1)}$  and $\mathcal T^{(2)}$ that satisfy Assumption \ref{assump}. 
For the case of $\alpha_1=2, \alpha_2=3, \beta_1=1, \beta_2=2$ and $L=3$ in \eqref{eq:cyclicmodel} (which is the case of $\alpha=2$ in Example \ref{ex:toy example 2}), the path in Figure \ref{fig:crn1} illustrates the path $(n,0)+\gamma_{\eta^0}$, where $\eta^0 \in \mathcal T^{(1)}$ consists of all dominant transitions and $\gamma_{\eta^0}$ is a cycle. The path in Figure \eqref{fig:crn2} illustrates the path $(n,0)+\gamma_{\eta^L}$, where $\eta^L \in \mathcal T^{(2)}$ consists of all dominant transitions except the $L$ th transition by which the associated Markov chain escapes the path $(n,0)+\gamma_{\eta^0}$, as highlighted by the red arrow in Figure \eqref{fig:crn2}. The precise construction of $\eta^0$ and $\eta^L$ are given in \eqref{eq:T1traj} and \eqref{eq:T2traj}, respectively. Furthermore, for the general models \eqref{eq:cyclicmodel}, the desired paths are explicitly constructed in Section \ref{sec:proof 2}.  



\begin{figure}[!h]
    \centering
\begin{equation*}
     \tikzset{state/.style={inner sep=1pt}}
   \begin{tikzpicture}[baseline={(current bounding box.center)}, scale=1]
   \node[state] (1) at (5,0)  {$(n,0)$};
   \node[state] (2) at (7,1)  {};
   \node[state] (3) at (8,2)  {$(n+3,2)$};
    \node[state] (4) at (7.8,1)  {$(n+2,1)$};
   \path[->]
    (1) edge node {} (2)
    (2) edge node {} (3)
    (3) edge node {} (1);
   \draw[->] (0,0.2) -- (10,0.2) node[right] {$A$};
   \draw[->] (0.2,0) -- (0.2,3) node[above] {$B$};
  \end{tikzpicture}
\end{equation*}
    \caption{The path $(n,0)+\gamma_{\eta^0}$ in Remark \ref{rem:simple cyclic}, where $\eta^0\in \mathcal T^{(1)}$.}
    \label{fig:crn1}
\end{figure}
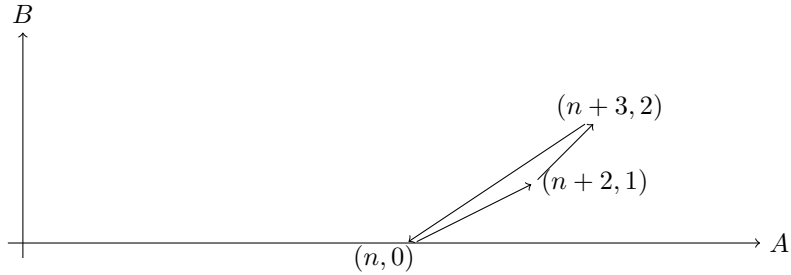

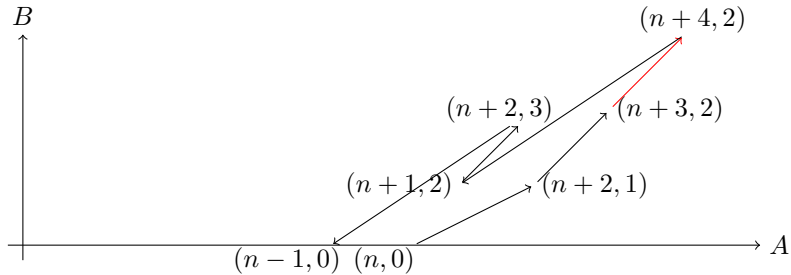
\begin{figure}[!h]
    \centering
  \begin{equation*}
     \tikzset{state/.style={inner sep=1pt}}
   \begin{tikzpicture}[baseline={(current bounding box.center)}, scale=1]
   \node[state] (1) at (5,0)  {$(n,0)$};
   \node[state] (2) at (7,1)  {};
   \node[state] (3) at (8,2)  {};
    \node[state] (4) at (9,3)  {};
  \node[state] (5) at (6,1)  {};
    \node[state] (6) at (7,2)  {$(n+2,3) \ \ \ \  \ \ \ \  $};
    \node[state] (7) at (4,0)  {$(n-1,0) \ \ \ \ \ $};
    \node[state] (2-1) at (7.8,1)  {$(n+2,1)$};
    \node[state] (3-1) at (8.8,2)  {$(n+3,2)$};
    \node[state] (4-1) at (9.1,3.2)  {$(n+4,2)$};
    \node[state] (5-1) at (5.2,1)  {$(n+1,2)$};
   \path[->]
    (1) edge node {} (2)
    (2) edge node {} (3)
    (3) edge[red] node {} (4)
    (4) edge node {} (5)
    (5) edge node {} (6)
    (6) edge node {} (7);
   \draw[->] (0,0.2) -- (10,0.2) node[right] {$A$};
   \draw[->] (0.2,0) -- (0.2,3) node[above] {$B$};
  \end{tikzpicture}
\end{equation*}
    \caption{The path $(n,0)+\gamma_{\eta^L}$ in Remark \ref{rem:simple cyclic}, where $\eta^L\in \mathcal T^{(2)}$.}
    \label{fig:crn2}
\end{figure}
\end{remark}




\bigskip

\subsection{Upper bound of a first passage time}\label{sec: upper bound}

In Theorem \ref{prop:cycle} and Theorem \ref{prop:cycle2}, we established lower bounds of the mixing times of the Markov chain associated with the reaction networks \eqref{eq:cyclicmodel}. How about an upper bound?
An upper bound of mixing times can typically be obtained using the Foster-Lyapunov criteria \cite{meyn1993stability}, a probabilistic coupling argument \cite{den2012probability, levin2017markov}, a geometric method based on path-decomposition \cite{saloff1997lectures, anderson2023new}, or spectral gap methods \cite{Han16}. Unfortunately, due to the special behavior of the continuous-time Markov chain  near the boundary of $\Z_{\ge 0}^2$, none of these seem to work easily.

For example, consider Example \ref{ex:toy example 2}. The generator $\mathcal A$ of the continuous-time Markov chain $X$ associated with \eqref{eq:example network} satisfies
\begin{align}\label{eq: generator}
    \mathcal A V(x)&= \lambda_1(x) \left (V(x+(\alpha,1)^\top)-V(x)\right )+\lambda_2(x)\left (V(x+(\alpha-1,1)^\top)-V(x)\right ) \notag\\
    &+
   \lambda_3(x)\left (V(x-(2\alpha-1,2)^\top)-V(x)\right ),
    \end{align}
For some suitable function $V:\Z^2_{\ge 0} \to \mathbb R_{\ge 0}$, suppose that
\begin{align}\label{eq:exp_ergo}
\mathcal{A}V(x) \le -aV(x)+b \quad \text{for all} \ \ x \in \mathbb{S},    
\end{align}
where $\mathcal A$ is the Markov generator of a continuous-time Markov process $X$ on a countable state space $\mathbb Z^2_{\ge 0}$. 
Then we have so-called exponential ergodicity that means $\Vert P^t(x)-\pi \Vert_{\text{TV}} \le C V(x) e^{-\eta t}$ for some constants $C>0$ and $\eta >0$ \cite{meyn1993stability}. This exponential ergodicity directly implies an upper bound of the mixing time as in \eqref{E:mixing_upper}.

Unfortunately, because of the special behavior at the boundary described in Example \ref{ex:toy example 2}, constructing a positive function $V$ satisfying \eqref{eq:exp_ergo} is challenging for the Markov chains associated with \eqref{eq:cyclicmodel}. To see this, we use the reaction network \eqref{eq:example network} showing later as a key example in Example \ref{ex:toy example 2}. First note that it is typical that when $|x| \gg 1$ identifying a dominant drift term in \eqref{eq: generator} is the key in the construction of a desired $V$ satisfying \eqref{eq:exp_ergo}. At each state of $(n,0)$, $(n+\alpha,1)$, and $(n+2\alpha-1,2)$, the dominant drift is given by reactions $\emptyset\to 2A+B, 2A+B\to 3A+2B, 3A+2B\to \emptyset$, respectively. Let $\lambda_1$, $\lambda_2$, and $\lambda_3$ denote the reaction intensities for these reactions given by \eqref{mass}, respectively. To achieve \eqref{eq:exp_ergo} at $x=(n,0)$, $(n+\alpha,1)$, and $(n+2\alpha-1,2)$,  we try to construct $V$ such that 
  \begin{align*}
     V(n+\alpha,1)-V(n,0)<0, \quad  V(n+2\alpha-1,2)-V(n+\alpha,1)<0,  \quad \text{and} \quad V(n,0)- V(n+2\alpha-1,1)<0.  
\end{align*}  
Obviously, these inequalities cannot hold at the same time.
Hence constructing a function $V$ satisfying \eqref{eq:exp_ergo} is not straightforward.

Due to this difficulty, rather than the upper bounds of the mixing times for $X=(X_A,X_B)$ associated with the reaction networks given in \eqref{eq:cyclicmodel}, in this section, we alternatively find the upper bound of the \textbf{first passage time}
\begin{equation}\label{Def:FPT}
    \tau_{C}=\inf\{t\in\R_{\ge 0} :\, \Vert X(t) \Vert_{2-1,\infty}=X_A(t) \leq C \},
\end{equation}
where $C\in(0,\infty)$ is a constant to be chosen.

The first passage time \eqref{Def:FPT} is closely related to the mixing time \cite{peres2015mixing}. By the coupling inequality \cite{den2012probability}, the total variation norm $\Vert P^t(x,\cdot)-\pi\Vert_{\text{TV}}$ is bounded above by the coupling probability $2P(X(t)=X'(t))$, where $X(t)$ and $X'(t)$ are the same Markov processes but $X'$ is initiated with the stationary distribution while $X$ is initiated at state $x$. When the stationary distribution is concentrated around the origin, then event 
$\{X(t)=X'(t)\}$ would take place when $X(t)$ is close to enough to the origin. 


\begin{prop}\label{prop:mean first hitting}
    Consider the reaction networks given in \eqref{eq:cyclicmodel}. 
If the choice of $\alpha_i$'s and $\beta_i$'s satisfy Assumption \ref{assump:cyclic}, then there exists $C>0$ such that the associated stochastic model satisfies 
\begin{align}\label{eq:mean first passage time}
  \mathbb  E_{(n,0)}[\tau_{C}] \le c n^{\theta} \quad \text{for some $c>0$,}
\end{align}
 where $\tau_{C}$ is as \eqref{Def:FPT} and $\theta$ is as in Theorem \ref{prop:cycle2}.
\end{prop}


\begin{remark}\label{rem:sharper fpt}
    For a sharper result, we can explore an upper bound of $\mathbb \E[\tau'_C]$, where
    \begin{align*}
        \tau'_C=\inf\{t\ge 0 : \Vert X(t)\Vert_\infty \le C\}.
    \end{align*}
    To do that, we need to handle the behavior of the model \eqref{eq:cyclicmodel} on the other boundary $\{x: x_1=0\}$. This task is challenging since the dominant reactions around this boundary are completely different from the previously used dominant reactions. Alternatively, we show the simulations about this first passage time in the next section.
\end{remark}

\subsection{Examples and simulation studies: mixing times and first passage times}\label{S:simulations}

In this section, we present simulation results for two examples of stochastic reaction networks which exhibit boundary-induced slow-mixing behaviors, as described in Remark \ref{rmk:slowmixing}. 
For simplicity we focus on dimension $d=2$. We write the coordinates of  the associated continuous-time Markov chain at time $t$, $X(t)$, as $(X_1(t),\,X_2(t))=(X_A(t),\,X_B(t))$ and think of them as the number of molecules for two chemical species $A$ and $B$.

For each example we shall estimate the \textbf{mixing time} $t^{0.2}_{\rm mix}(x)$ (defined in \eqref{E:mixing}) and the \textbf{first passage time}
\begin{equation*}
    \tau'_{5}=\inf\{t\in\R_{\ge 0} : \Vert X(t) \Vert_{\infty} \leq 5 \},
\end{equation*}
defined in Remark \ref{rem:sharper fpt}. We shall estimate the mean first passage time by averaging over $100$ independent trajectories. To estimate the mixing time starting at $x$, we generate $100$ i.i.d. sample trajectories $\{X^{(i)}_t\}_{i=1}^{100}$ of $X$ starting at $x$ using the Gillespie's algorithm \cite{gillespie1976general}. The probability $P^t(x,y)=\P_x(X_t=y)$ is approximated by the empirical distribution at time $t$, namely 
\[
    P^t(x,y) \approx \frac{1}{100} \sum_{i=1}^{100} \mathbf{1}_{\{X_t^{(i)} = y\}}, 
\]
where $\mathbf{1}_A$ is the indicator of event $A$. 
This approximation is computed for $y\in [0,100]^2$ and then inserted into the approximation of total variation below: 
\begin{align}
    \Vert P^t(x,\cdot) - \pi \Vert_{\rm TV} =& \frac{1}{2} \sum_{y\in \Z^2_{\geq 0}} \vert P^t(x,y) - \pi(y) \vert \label{eq:tv and l1}\\
    \approx& \frac{1}{2} \sum_{y\in [0,100]^2 } \vert P^t(x,y) - \pi(y) \vert + \frac{1}{2} \sum_{y\not \in [0,100]^2 } P^t(x,y) \notag\\
    =& \frac{1}{2} \sum_{y\in [0,100]^2 } \vert P^t(x,y) - \pi(y) \vert + \frac{1}{2} \left(1-\sum_{y\in [0,100]^2 } P^t(x,y)    \right).\label{def:mixing_sim}
\end{align}
The derivation of the identity \eqref{eq:tv and l1} can be found in \cite{levin2017markov}. The approximation above is accurate in our simulations since $\pi$ captures most of the mass of the stationary distribution in the set $[0,100]^2$.  For computational efficiency, we estimate the mixing time in \eqref{E:mixing} by computing the total variation distance  $ \Vert P^t(x,\cdot) - \pi \Vert_{\rm TV}$ for $t$ being multiples of 100 until it falls below the threshold $\varepsilon = 0.2$.

\begin{example} \label{ex: toy example}
Consider the reaction network 
\begin{align}\label{model22}
    B \xrightleftharpoons[1]{1} 2B, \quad \emptyset \xrightleftharpoons[1]{1} A+B.
\end{align}
Four reactions are involved in this network, whose reaction vectors are denoted by 
\[
    \eta_1 = \begin{pmatrix}
    0 \\ 1
    \end{pmatrix},\qquad  \eta_2 = \begin{pmatrix}
    0 \\ -1
    \end{pmatrix},\qquad  \eta_3 = \begin{pmatrix}
    1 \\ 1
    \end{pmatrix},\qquad  \eta_4 = \begin{pmatrix}
    -1 \\ -1
    \end{pmatrix}.
\]
We will show that the corresponding continuous-time Markov process $X$ satisfies all assumptions in Corollary \ref{cor:mixing}, and the 
mixing time of $X$ is of order $n^2$ when the process starts at $(n,0)$, as $n\rightarrow \infty$.

First, we verify Assumption \ref{assump}, with $\theta_1=1$, $\theta_2=2$ and the following sequences 
\begin{align*}
    \mathcal{T}^{(1)}  = \left\{ \eta^1 = (\eta_3, \eta_4),  \eta^2 = (\eta_3,\eta_3, \eta_4,\eta_4) \right\},   \quad
    \mathcal{T}^{(2)}  = \left\{ \eta^3 = (\eta_3, \eta_1,\eta_4, \eta_4) \right\}.
\end{align*}
Figures \ref{fig:crn3} and \ref{fig:crn3-2} visualize paths $(n,0)+\eta^i$ for $i=1,2,3$.
\begin{figure}[!h]
\centering
    \begin{equation}
     \tikzset{state/.style={inner sep=1pt}}
   \begin{tikzpicture}[baseline={(current bounding box.center)}, scale=1]
   \node[state] (1) at (3,0)  {$(n,0)$};
   \node[state] (2) at (4,1)  {$(n+1,1)$};
   \node[state] (11) at (10,0)  {$(n,0)$};
   \node[state] (12) at (11,1)  {};
   \node[state] (13) at (12,2)  {$(n+2,2)$};
    \node[state] (14) at (12,1)  {$(n+1,1)$};
   \path[->]
    (1) edge[bend right=10] node {} (2)
    (2) edge[bend right=10] node {} (1)
    (11) edge[bend right=10] node {} (12)
    (12) edge[bend right=10] node {} (13) 
    (13) edge[bend right=10] node {} (12) 
    (12) edge[bend right=10] node {} (11); 
   \draw[->] (0,0.2) -- (5,0.2) node[right] {$A$};
   \draw[->] (0.2,0) -- (0.2,3) node[above] {$B$};
   \draw[->] (7,0.2) -- (12,0.2) node[right] {$A$};
   \draw[->] (7.2,0) -- (7.2,3) node[above] {$B$};
  \end{tikzpicture}
\end{equation}
    \caption{The paths $(n,0)+\eta^1$ and $(n,0)+\eta^2$ for \eqref{model22} }
    \label{fig:crn3}
\end{figure}
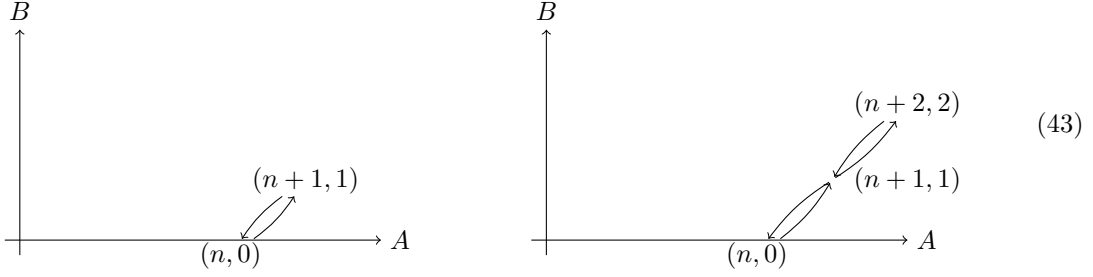

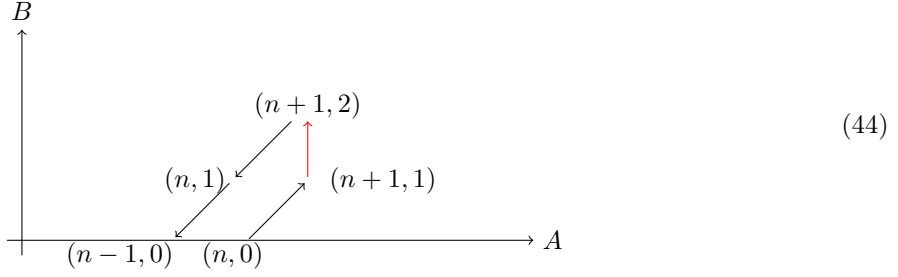
\begin{figure}[!h]
    \begin{equation}
     \tikzset{state/.style={inner sep=1pt}}
   \begin{tikzpicture}[baseline={(current bounding box.center)}, scale=1]
   \node[state] (1) at (3,0)  {$(n,0)$};
   \node[state] (2) at (4,1)  {};
   \node[state] (3) at (4,2)  {$(n+1,2)$};
    \node[state] (4) at (3,1)  {};
   \node[state] (5) at (2.2,0.2)  {};
    \node[state] (7) at (5,1)  {$(n+1,1)$};
    \node[state] (8) at (2.5,1)  {$(n,1)$}; \node[state] (9) at (1.5,0)  {$(n-1,0)$};
   \path[->]
    (1) edge node {} (2)
    (2) edge[red] node {} (3) 
    (3) edge node {} (4) 
    (4) edge node {} (5); 
   \draw[->] (0,0.2) -- (7,0.2) node[right] {$A$};
   \draw[->] (0.2,0) -- (0.2,3) node[above] {$B$};
  \end{tikzpicture}
\end{equation}
    \caption{The path $(n,0)+\eta^3$ for \eqref{model22} }
    \label{fig:crn3-2}
\end{figure}

Starting with initial condition $(n,0)$, the path $(n,0) \rightarrow (n+1,1) \rightarrow (n,0)$ occurs with large probability,  which can also be specified by the sequence of transitions $\eta^1 = (\eta_3, \eta_4)$. A possible rare event occurs when the second transition follows a different reaction than $\eta_4$, either $\eta_1$ or $\eta_3$. In either cases, the trajectory return to the $x-$axis by two firings of $\eta_4$. Probability of these paths can be computed by
\begin{align*}
    \P_{(n,0)}\left( E_{\eta^1} \right ) =  &\, \frac{n+1}{n+3} , \quad     \P_{(n,0)}\left( E_{\eta^2} \right ) =  \, \frac{1}{n+3} \cdot \frac{2n+4}{2n+9} \cdot \frac{n+1}{n+3} , \quad    \P_{(n,0)}\left( E_{\eta^3} \right ) =  \, \frac{1}{n+3} \cdot \frac{2n+2}{2n+7} \cdot \frac{n}{n+2} , 
\end{align*}
In particular, since both $\gamma_{\eta^1}$ and $\gamma_{\eta^2}$ are cyclic paths, Hence $\mathcal{T}^{(1)} = \{\eta^1, \eta^2\}$ with 
\begin{align*}
     \P_{(n,0)}\left( \left(E_{\eta^1} \cup E_{\eta^2}\right)^c\right ) & = 1 - \frac{n+1}{n+3}-\frac{1}{n+3} \cdot \frac{2n+4}{2n+9} \cdot \frac{n+1}{n+3} \\
    & = \frac{2}{n+3} \left( 1 - \frac{n+2}{2n+9} \cdot \frac{n+1}{n+3}\right) \leq \frac{2}{n}.
\end{align*}
Equation \eqref{eq:assump_1st} in Assumption \ref{assump} is verified with $\theta_1 = 1$, because 
\[
   \P_{(n,0)}\left( \bigcup_{\eta \in \mathcal{T}^{(1)}} E_{\eta} \right ) =   1 - \P_{(n,0)}\left( \left(E_{\eta^1} \cup E_{\eta^2}\right)^c\right )  \geq 1 - \frac{2}{n}, 
\]
with $\theta_1 = 1$. Moreover, 
\begin{align}
    \P_{(n,0)}\left( \left(E_{\eta^1} \cup E_{\eta^2} \cup E_{\eta^3}\right)^c\right ) & = 1 - \frac{n+1}{n+3}-\frac{1}{n+3} \cdot \frac{2n+4}{2n+9} \cdot \frac{n+1}{n+3}  - \frac{1}{n+3} \cdot \frac{2n+2}{2n+7} \cdot \frac{n+1}{n+3}  \notag \\
    & = \frac{2}{n+3} \left( 1 - \frac{n+2}{2n+9} \cdot \frac{n+1}{n+3} - \frac{n+1}{2n+7} \cdot \frac{n+1}{n+3} \right)  \notag \\
    & = \frac{2}{(n+3)^2}\cdot \frac{2n^2 + 27n + 60}{(2n+9)(2n+7)}
    \leq \frac{2}{n^2}. \notag
\end{align}
Hence, if $\mathcal{T}^{(2)} = \{\eta^3\}$, then
\begin{align*}
    \P_{(n,0)}\left( \bigcup_{\eta \in \mathcal{T}^{(1)} \cup \mathcal{T}^{(2)}}   E_{\eta} \right )=  1 - \P_{(n,0)}\left( \left(E_{\eta^1} \cup E_{\eta^2} \cup E_{\eta^3}\right)^c\right )   \geq 1 - \frac{2}{n^2},
\end{align*}
which verifies equation \eqref{eq:assump_3rd}  with $\theta_2 = 2$ in Assumption \ref{assump}.

Next, we verify  Assumptions \ref{assump2} and  \ref{A:Positiveecurrent}.
Note that any trajectory originated on the axis $\{x_2 = 0\}$ can only leave after the occurrence of reaction $\emptyset \rightarrow A + B$, which takes a unit exponential time, hence Assumption \ref{assump2} is  satisfied with $\kappa=1$.  The reaction network also meets the assumptions in Theorem \ref{thm:def0}, hence there is a unique stationary distribution which verifies Assumption \ref{A:Positiveecurrent}. 

In conclusion, the continuous-time Markov process $X$ of the reaction network \eqref{model22} satisfies Assumption \ref{assump}, Assumption \ref{assump2} and Assumption \ref{A:Positiveecurrent}. Hence Corollary \ref{cor:mixing} applies  and the mixing time of $X$ is at least of order $n^2$ when the process starts at $(n,0)$, as $n\rightarrow \infty$.

\textbf{Simulation of stochastic model in \eqref{model22}.} Stochastic simulations of model \eqref{model22} are plotted in Figure \ref{fig:2}. In particular, boundary-induced slow mixing, described in Remark \ref{rmk:slowmixing}, can be observed in Figure \ref{fig:toy2_traj}, where the process swiftly approach the boundary, stays in close proximity to the boundary for a long time and eventually move towards the bulk of the state space. Note that such behavior can only be observed when there are little B species present as in the initial condition. Figure \ref{fig:xd} presents log-log plot of mean first passage time and mixing time, respectively, against initial conditions $(n,0)$ for $n$ between 100 and 3000 (assuming no B species initially). 
Figure \ref{fig:xd} demonstrates that when the process begins at the initial condition $(n,0)$, both the mean first passage time and mixing time are of order $n^2$ as the slopes of the straight lines are nearly $2$. Furthermore, the matching slopes confirms that our lower bound from Corollary \ref{cor:mixing} and upper bound from Proposition \ref{prop:mean first hitting} are sharp for the model in \eqref{model22}. 

\begin{figure}[!htbp]
\begin{subfigure}{0.45\textwidth}
\includegraphics[width=\linewidth, height=5cm]{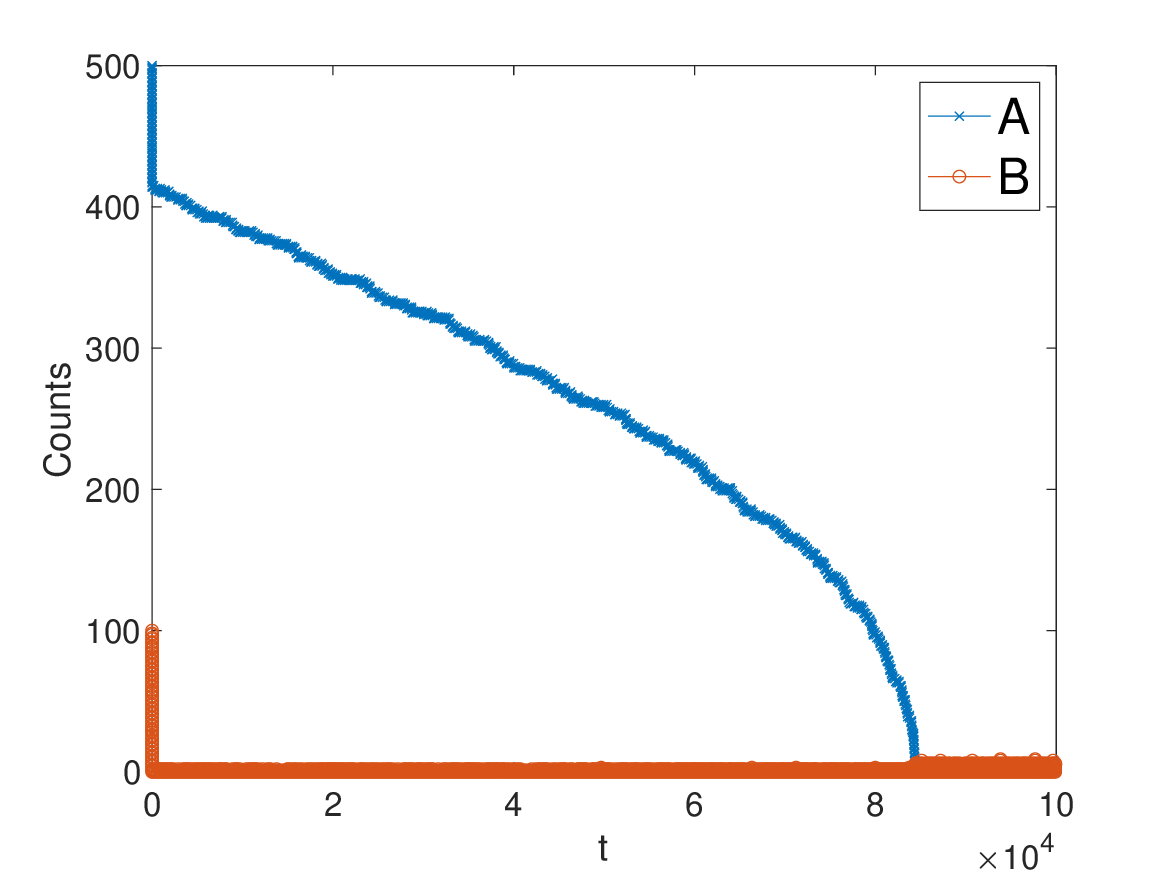}
\caption{Sample trajectory starting at $(500, 100)$}
\label{fig:toy2_traj}
\end{subfigure}
\begin{subfigure}{0.45\textwidth}
\includegraphics[width=\linewidth, height=5cm]{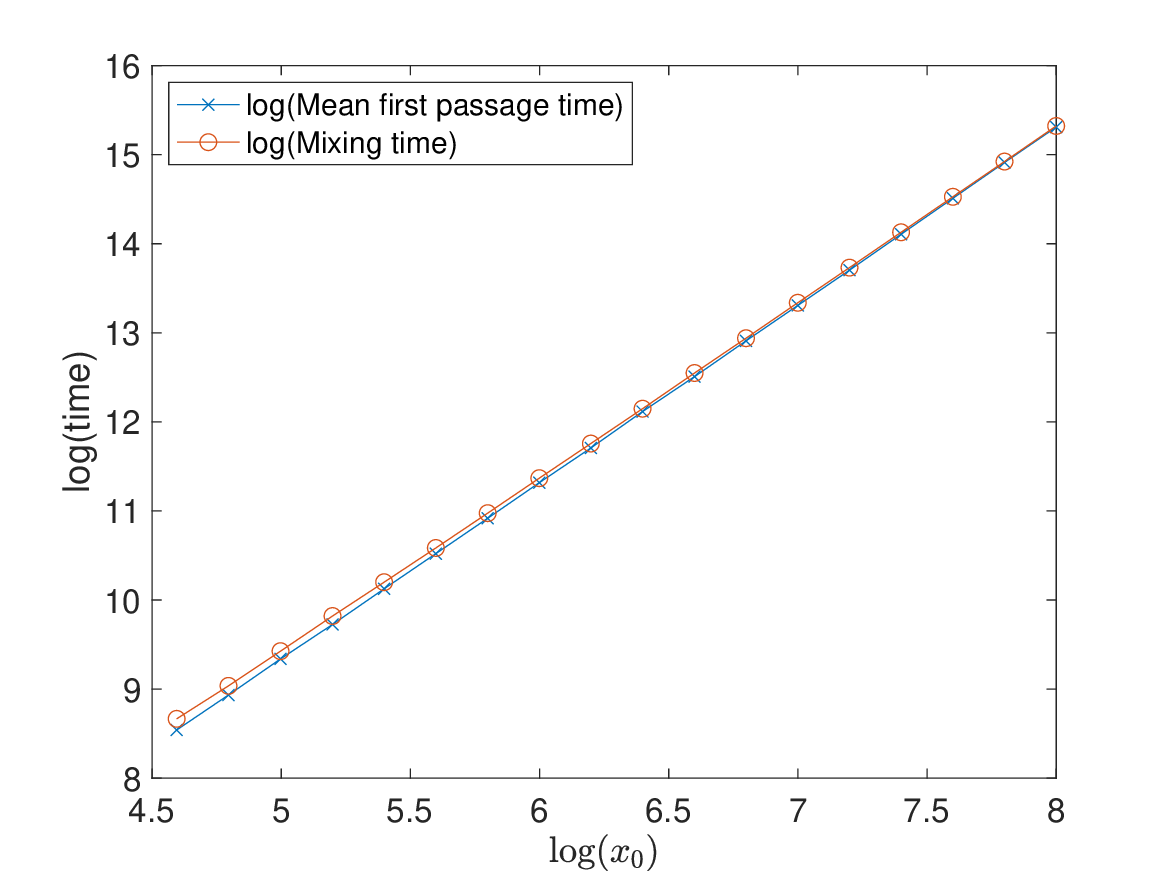}
\caption{Comparison of mean first passage time and mixing time.}
\label{fig:xd}
\end{subfigure}
\caption{Stochastic simulations for the continuous-time Markov process for the toy model in \eqref{model22}. Sample trajectory is plotted in Figure \ref{fig:toy2_traj} for initial condition  $x_0 = (500,100)$. In Figure \ref{fig:xd}, mean first passage time and mixing time, averaged over 100 trajectories, is plotted against initial condition $x_0 = (n,0)$ in the log-log scale.}
\label{fig:2}
\end{figure}

\end{example}

\begin{example}\label{ex:toy example 2}
For a positive integer $\alpha$, consider a family of continuous-time Markov chains, indexed by $\alpha$, given by the reaction network
    \begin{align}\label{eq:example network}
        \begin{split}
        &\emptyset \xrightarrow{\qquad \kappa_1 \qquad }  \alpha A+B \\
        &\nwarrow{ \kappa_3 } \qquad \quad \swarrow{ \kappa_2 }\\
        & \quad (2\alpha -1)A+2B
    \end{split}
    \end{align}
\end{example}
Three reactions are involved in this network, whose reaction vectors are denoted by 

\[
    \eta_1 = \begin{pmatrix}
    \alpha \\ 1
    \end{pmatrix},\qquad  \eta_2 = \begin{pmatrix}
    \alpha-1 \\ 1
    \end{pmatrix},\qquad  \eta_3 = \begin{pmatrix}
    -2\alpha+1 \\ -2
    \end{pmatrix}.
\]
Note this network is a special case of the reaction network in \eqref{eq:cyclicmodel} that satisfies Assumption \ref{assump:cyclic} with $\alpha_1 = \alpha$, $\alpha_2 = 2\alpha-1$,  $\beta_1 = 1$, and $\beta_2 = 2$.
 By Theorem \ref{prop:cycle2},  the conclusions of Theorem \ref{thm}, Corollary \ref{cor:mixing} and Corollary \ref{Cor:Thm} all hold with $\theta = 1 + \theta_1 =  1+\alpha-1 = \alpha$, regardless the value of $\alpha$, and the mixing time of $X$ is at least of order $n^\alpha$ when the process starts at $(n,0)$, as $n\rightarrow \infty$.

\textbf{Simulation of stochastic model in \eqref{eq:example network} with $\alpha = 2$.} In the following simulations, we assume all reaction rate constant $\kappa_i$ to be 1, hence the stationary distribution is given by Theorem \ref{thm:def0} with $c = (1,1)$. 

\begin{figure}[!htbp]
\centering
\begin{subfigure}{0.45\textwidth}
\includegraphics[width=\linewidth, height=5cm]{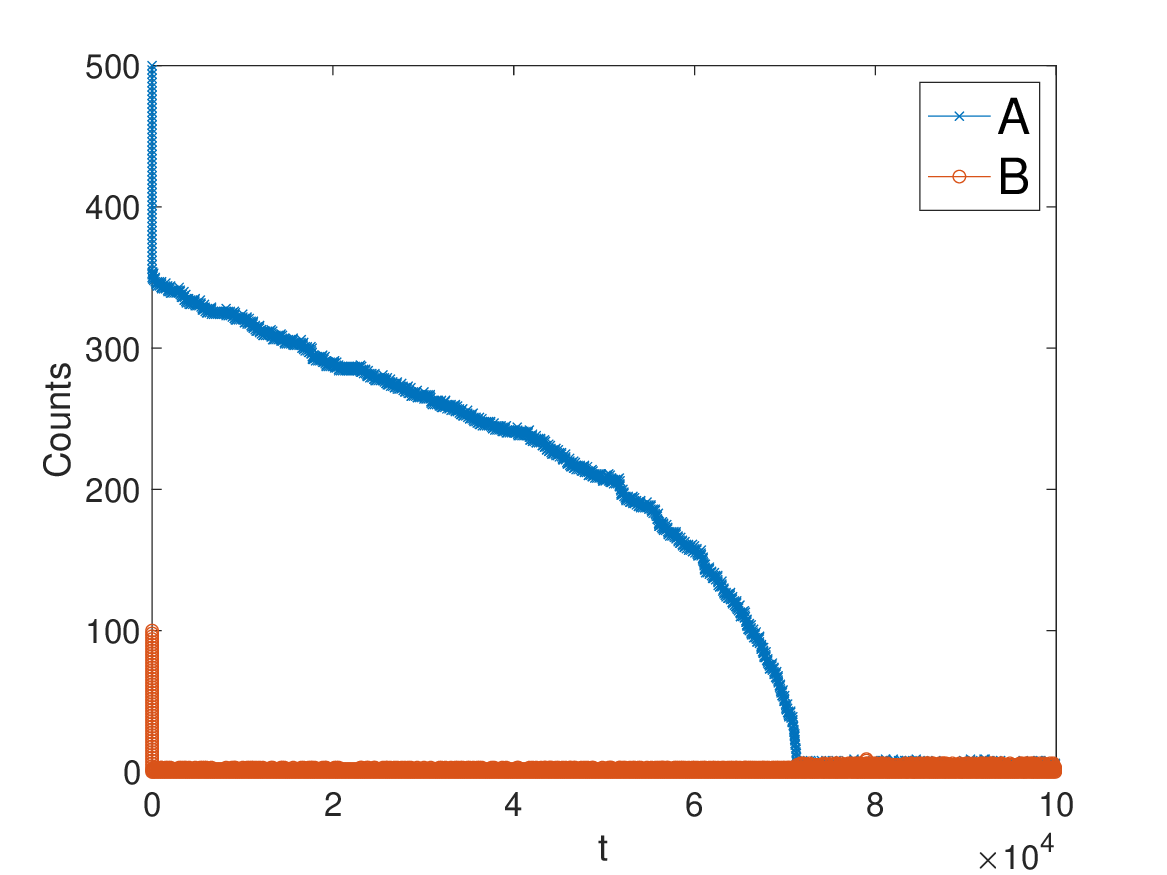}
\caption{Sample trajectory starting at $(500, 100)$}
\label{fig:toy1_traj_1}
\end{subfigure}
\begin{subfigure}{0.45\textwidth}
\includegraphics[width=\linewidth, height=5cm]{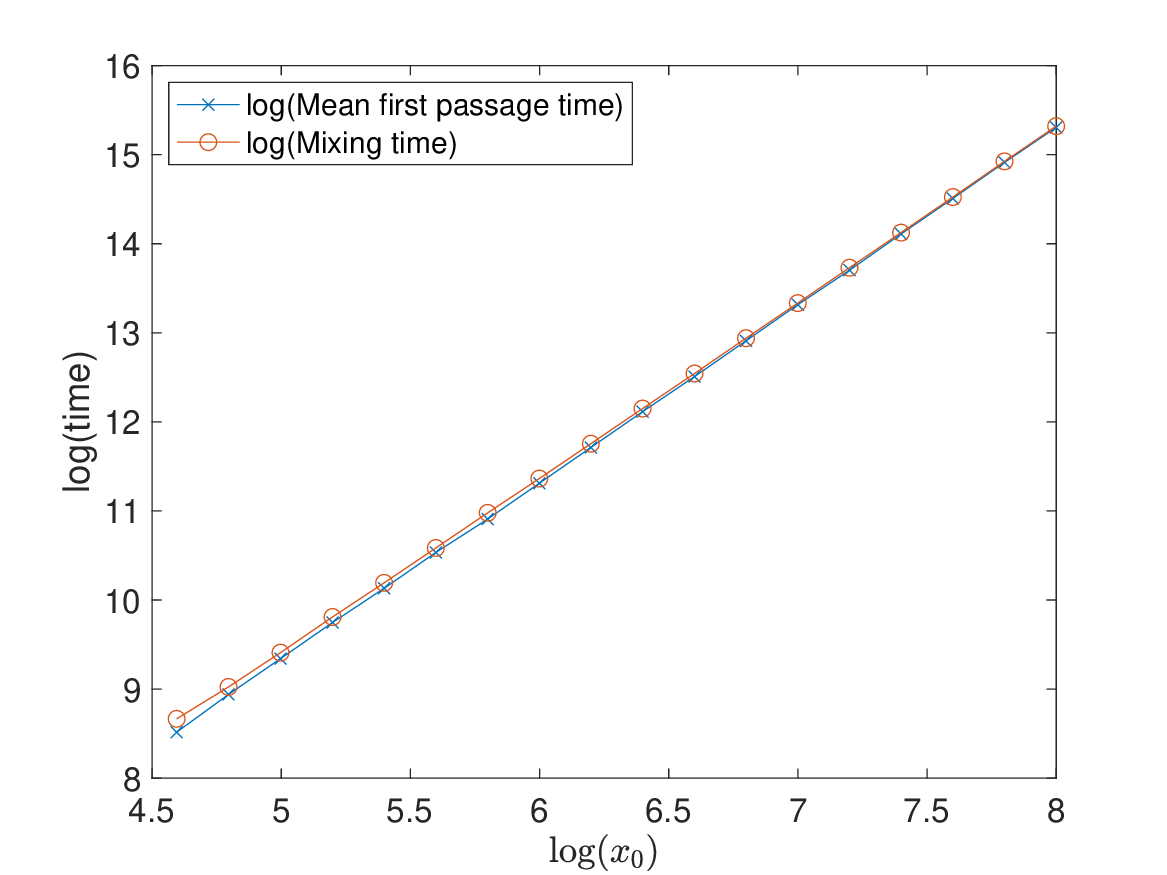}
\caption{Mean first passage time against initial position}
\label{fig:toy1_fpt_mt}
\end{subfigure}
\caption{Stochastic simulations of the continuous-time Markov process for the model in \eqref{eq:example network} with $\alpha =2$. In Figure \ref{fig:toy1_traj_1}, sample trajectories are plotted for the initial condition  $x_0 = (500,100)$. In Figure \ref{fig:toy1_fpt_mt}, the mean first passage times and mixing times, averaged over 100 trajectories, are plotted against initial condition $x_0 = (n,0)$ in the log-log scale.}
\label{fig:1}
\end{figure}

Stochastic simulations of model \eqref{eq:example network} with $\alpha =2$ are plotted in Figure \ref{fig:1}. In Figure \ref{fig:toy1_traj_1}, sample trajectory is plotted with the initial condition $x_0 = (500, 100)$. In particular, boundary-induced slow mixing, described in Remark \ref{rmk:slowmixing}, can be observed in Figure \ref{fig:toy2_traj} and the process swiftly approaches the boundary, stays in close proximity to the boundary for a long time. Figure \ref{fig:toy1_fpt_mt} show log-log plot of the mean first passage times and the mixing times, respectively, against initial conditions $(n,0)$ for $n$ between 100 and 3000 (assuming no B species initially).  Figure \ref{fig:toy1_fpt_mt} demonstrate that when the process begins at the initial condition $(n,0)$, both the mean first passage time and mixing time are of order $n^2$ as the slopes of the corresponding straight lines are nearly $2$. Furthermore, the matching slopes confirms that our lower bound from Corollary \ref{cor:mixing} and upper bound from Proposition \ref{prop:mean first hitting} are sharp for the model in \eqref{eq:example network}.




\FloatBarrier

\section{Proof of Theorem \ref{thm}}\label{sec:proof 1}

We give the proof of Theorem \ref{thm} in this section. 
Fix $\delta \in (0,1)$, it suffices to find a constant $C\in (0,\infty)$ and $N_{\delta}\in (0,\infty)$, such that for any $n\geq N_{\delta}$ and any initial state $x\in \mathbb{I}_n$, 
\begin{align}\label{eq:thm1_also}
  \sup_{t\in[0,\,Cn^{\theta}]}  \P_{x}\left(\|X(t)\|_{d-1,\infty}\leq \frac{n}{2}  \right)\leq 1-\delta. 
\end{align}
To prove \eqref{eq:thm1_also},
we denote by 
\begin{equation}\label{Def:Zi}
Z(i):=
\Big(X_1(\nu_i),X_2(\nu_i),\cdots, X_{d-1}(\nu_i) \Big)\in \Z_{\ge 0}^{d-1},
\end{equation}
the location of the $i$ th visit to the boundary $\{x_d=0\}$, where $\nu_i$'s are the $i$ th time for visiting the axis defined as \eqref{eq:StoppingDef}.
Note that $\{Z(i)\}_{i\in\Z_{\ge 0}}$ is itself a discrete-time Markov chain on $\Z^{d-1}_{\ge 0}$. 
The first passage time for the reduced process,
\begin{equation}\label{Def:tauZ}
    \tau_{k}^Z := \inf \{ i\in \Z_{\ge 0}: \, \|Z(i)\|_{\infty} \leq k\}= \inf \left\{ i\in \Z_{\ge 0}: \, \max_{1\leq j\leq d-1}X_j(\nu_i)\, \leq\, k\right\},
\end{equation}
the first time-step for  $Z$ to be less than or equal to $k$,
will play a key role in our proof of \eqref{eq:thm1_also}.

\medskip

The proof of \eqref{eq:thm1_also} begins as follows. Let $N(t)$ be the total number of times that the $d$ th coordinate of $X(t)$ becomes $0$ during $[0,t]$, that is, 
\begin{align}\label{eq:CountDef}
    N(t)  \coloneqq \max \{i\in \Z_{\ge 0} : \nu_i \leq t \}. 
\end{align}
For any  constants $\theta,\,b\in (0,\infty)$, $c_*>-n/2$ and any $x\in \mathbb{I}_n$, 
\begin{align}
&\P_{x}\left(\|X(t)\|_{d-1,\infty}\leq \frac{n}{2}  \right) \notag \\
= & \,\P_{x}\left(\|X(t)\|_{d-1,\infty}\leq \frac{n}{2}, N(t) \geq \lfloor bn^{\theta} \rfloor \right)\,+\,\P_{x}\left(\|X(t)\|_{d-1,\infty}\leq \frac{n}{2}, N(t) < \lfloor bn^{\theta} \rfloor \right) \notag\\
\leq &\,\P_{x}\left( N(t) \geq \lfloor bn^{\theta} \rfloor\right)\,+\,\P_{x}\left( \|X(t)\|_{d-1,\infty}\leq \frac{n}{2},\, N(t) < \lfloor bn^{\theta} \rfloor, \,\tau^Z_{n/2\,+c_* } \leq \lfloor bn^{\theta} \rfloor   \right) \notag \\
 & \qquad \qquad\qquad \qquad\qquad \quad +\, \P_{x}\left( \|X(t)\|_{d-1,\infty}\leq \frac{n}{2},\, N(t) < \lfloor bn^{\theta} \rfloor,\, \tau^Z_{n/2\,+c_* } > \lfloor bn^{\theta} \rfloor  \right) \notag \\
\leq &\,\underbrace{\P_{x}\left(  N(t) \geq \lfloor bn^{\theta} \rfloor\right) }_{\text{many visits to the boundary}}\,+\, \underbrace{\P_{x}\left( \tau^Z_{n/2\,+c_* } \leq \lfloor bn^{\theta} \rfloor  \right)}_{\text{first passage time estimate}}\label{eq:thm1_also_0} 
 \\
 & \qquad \qquad\qquad \qquad\qquad \quad +\, \underbrace{\P_{x}\left( \|X(t)\|_{d-1,\infty}\leq \frac{n}{2},\, N(t) < \lfloor bn^{\theta} \rfloor , \, \tau^Z_{n/2\,+c_* } > \lfloor bn^{\theta} \rfloor \right)}_{\text{Z stayed above $n/2+c_*$ but $\Vert X(t) \Vert_{d-1,\infty}$ doesn't}}.
\label{eq:thm1_also_general}
\end{align}
For the rest of the proof, we choose $b$ and $c_*$ suitably and bound each of the three terms in \eqref{eq:thm1_also_0} and \eqref{eq:thm1_also_general}.
We shall take $b$ small enough (according to the sentence after \eqref{eq:thm1_also_2_general_2}), and we shall take $c_*$ to be any constant  strictly larger than 
\begin{align}\label{para:pathconstant}
\max_{\eta \in \mathcal{T}^{(1)} \cup \mathcal{T}^{(2)}} \,\max_{1\leq j\leq |\eta|} \left\| \sum_{i=1}^{j} \eta_i \right\|_{d-1,\infty},
\end{align}
the maximum change among the first $d-1$ coordinates of the Markov chain  along the path $\gamma_{\eta}$ defined in \eqref{Def:gammaeta}, among all $\eta\in \mathcal{T}^{(1)} \cup \mathcal{T}^{(2)}$.

\begin{remark}\rm
The assumption of $\mathcal{T}^{(2)}$ being finite can be relaxed. As long as the number defined in \eqref{para:pathconstant} is finite,  $\mathcal{T}^{(2)}$ can be an infinite subset and all our proofs and results in this paper still hold.
\end{remark}

\subsection{Immediate consequences of  assumptions \ref{assump}-\ref{assump2}}\label{sec:assumptions}

First, we give some immediate consequences of  assumptions \ref{assump}-\ref{assump2}. Assumption \ref{assump2} enables us to control $N(t)$ in \eqref{eq:CountDef}, which gives the following bound for the first term on the right of \eqref{eq:thm1_also_0}.

\begin{lemma}[Number of visits to the boundary]\label{lm:unit_Poisson_bdd}

Suppose  Assumption \ref{assump2} holds.
For any constants $\theta,\,b\in(0,\infty)$, $t\in\R_{\ge 0}$ and $n\in\mathbb{N}$, and 
any initial state $x\in \mathbb{I}_n$, 
\begin{align}\label{eq:uppbdd_1}
\P_{x}\left( N(t) \geq \lfloor bn^{\theta} \rfloor\right)\leq \frac{\kappa t}{\lfloor bn^{\theta} \rfloor}.
\end{align}
\end{lemma}

\begin{proof}
The return times are, by definition $\{\nu_i-\nu_{i-1}\}_{i\in\mathbb{N}}$.
Assumption \ref{assump2} and the strong Markov property of $X$ implies that the process $N$ is stochastically dominated above by  a Poisson process $Y$ with rate $\kappa$.
Therefore,
\begin{align*}
\P_{x}\left(N(t)\geq \lfloor bn^{\theta} \rfloor \right) \leq &\, \frac{\E_{x}\left[N(t)\right]}{\lfloor bn^{\theta} \rfloor}  \leq \frac{ \E_{x}\left[Y(t)\right]}{\lfloor bn^{\theta} \rfloor} = \frac{\kappa t}{\lfloor bn^{\theta} \rfloor}.
\end{align*}
\end{proof}

\begin{remark}[Holding time versus return time]\rm
Note that the return time $\nu_i-\nu_{i-1}$ is equal to $(\nu_i-\mu_i)+(\mu_i-\nu_{i-1})$, where $\mu_i-\nu_{i-1}$ can be read as a holding time at the boundary $\{x_d=0\}$. For our applications introduced in Section \ref{sec:class of reaction networks}, we verify Assumption \ref{assump2} by using the fact that the holding time $\mu_i-\nu_{i-1}$ follows the same exponential distribution at any state at the boundary.
\end{remark}

Assumption \ref{assump}(ii) gives  control to the transition probabilities of the reduced, $d-1$-dimensional process $Z$ defined in \eqref{Def:Zi}. 
For $\eta\in \mathcal{T}^{(2)}$,  the endpoint of the path $\gamma_{\eta}$ is $\sum_{i=1}^{|\eta|}\eta_i=(x^{\eta}_1,x^{\eta}_2,\ldots,x^{\eta}_{d-1},\;0)$ for
a unique element $(x^{\eta}_1,x^{\eta}_2,\ldots,x^{\eta}_{d-1})\in \Z^{d-1}\backslash \{ {\bf 0}\}$. 
Let $J^{(2)}\subset \Z^{d-1}\backslash \{ {\bf 0}\}$ be the set of endpoint coordinates of paths in $\mathcal{T}^{(2)}$; that is,
\begin{equation}\label{Def:J2}
    J^{(2)} := \left\{(x^{\eta}_1,x^{\eta}_2,\ldots,x^{\eta}_{d-1})\in \Z^{d-1}\backslash \{ {\bf 0}\}:\, (x^{\eta}_1,x^{\eta}_2,\ldots,x^{\eta}_{d-1},\,0)=\sum_{i=1}^{|\eta|}\eta_i \text{ for some } \eta \in \mathcal{T}^{(2)} \right\}.
\end{equation}

\begin{lemma}[Transition probabilities for reduced lazy walk]\label{L:2Dto1D}
Suppose Assumption \ref{assump}  holds and $\nu_i<\infty$ for all $i\in\Z_{\ge 0}$ almost surely under $\P_x$ for all $x\in \{x_d=0\}\cap \Z_{\ge 0}^d$.
The process $\{Z(i)\}_{i\in \Z_{\ge 0}}$ defined in \eqref{Def:Zi} is a  discrete-time Markov chain on  $\Z^{d-1}_{\ge 0}$ whose transition probabilities, denoted by $p_{a_1,a_2} := \P \left( Z(1) = a_2\,|\, Z(0) = a_1\right)$, satisfy 
\begin{align}\label{eq:reduction_assumption_Lazy}
p_{a,a} \geq 1 - \dfrac{c_1}{\|a\|_{\infty}^{\theta_1}} 
\end{align} 
and
\begin{align}\label{eq:reduction_assumption}
\sum_{j'\not \in J^{(2)} \cup  \{{\bf 0}\} } p_{a,a+j'} \leq \frac{c_2}{\|a\|_{\infty}^{\theta_2}} 
\end{align} 
for all $a \in \Z_{\ge 0}^{d-1}$ such that  $\|a\|_{\infty} \geq N_0$, where $N_0$ is as in Assumption \ref{assump}. 
\end{lemma}

\begin{proof}
This follows directly from Assumption \ref{assump} because under such $\P_{(a,0)}$, we have 
\begin{align}
    \{Z(1)-Z(0)={\bf 0}\} & \supset   \bigcup_{\eta \in \mathcal{T}^{(1)}} E_{\eta} \qquad \text{and}  \\
   \{Z(1)-Z(0)\not \in  J^{(2)}\cup \{ {\bf 0}\}\} & \subset   \left( \bigcup_{\eta \in \mathcal{T}^{(1)} \cup \mathcal{T}^{(2)}} E_{\eta} \right)^c.
\end{align}
\end{proof}





\subsection{Rare excursions from the boundary}\label{sec: excursion}

In this section, we bound the term on the right of \eqref{eq:thm1_also_general}, namely 
$$\P_{x}\left( \|X(t)\|_{d-1,\infty}\leq \frac{n}{2}, N(t) < \lfloor bn^{\theta} \rfloor ,   
\tau^Z_{n/2\,+c_* } > \lfloor bn^{\theta} \rfloor
\right),$$
where $Z(i)$ is defined in \eqref{Def:Zi}, and  $c_*$ is any fixed constant strictly larger than 
the number defined in \eqref{para:pathconstant}. 

Our key observation is that the event in the above display is a \textit{rare} excursion  of $X$ from the boundary $\{x_d=0\}$. This event 
can \textit{not} be obtained by any trajectory $\gamma_{\eta}$ for any $\bigcup_{\eta \in \mathcal{T}^{(1)} \cup \mathcal{T}^{(2)}} E_{\eta}$ and we can bound the probabilities of these dominating excursions using Assumption \ref{assump}.

\begin{lemma}\label{lm:rare_excursion_general}
Suppose Assumption \ref{assump}  holds and $\nu_i<\infty$ for all $i\in\Z_{\ge 0}$ almost surely under $\P_x$ for all $x\in \{x_d=0\}\cap \Z_{\ge 0}^d$.
Let $c_*$ be a fixed constant that is strictly larger than the number  in \eqref{para:pathconstant}. 
For any constants $b\in (0,\infty)$ and $t\in\R_{\ge 0}$, and any initial state $x\in \mathbb{I}_n$ and  $\theta\in (0,\infty)$,
\begin{align}\label{eq:uppbdd_2}
   \P_{x}\left( \|X(t)\|_{d-1,\infty}\leq \frac{n}{2},\; N(t) < \lfloor bn^{\theta} \rfloor , \; \tau^Z_{n/2\,+c_* } > \lfloor bn^{\theta} \rfloor  \right) \leq b c_2 \, 2^{\theta_2}\,n^{\theta-\theta_2},
\end{align}
for all $n> 2(N_0+c_*)$, where $c_2$ is the constant in Assumption \ref{assump}(ii). 
\end{lemma}

\begin{proof}
For any constant $b\in(0,\infty)$ and $t\in \R_{\ge 0}$, since 
$\nu_i<\infty$ for all $i\in\Z_{\ge 0}$ almost surely,
\[
\{N(t)<\lfloor bn^{\theta} \rfloor\}= \{\nu_{\lfloor bn^{\theta} \rfloor}>t\}=\bigcup_{j=0}^{\lfloor bn^{\theta} \rfloor -1}\{\nu_j \leq t<\nu_{j+1}\}.
\]
Also, for all $0\leq k<n$ and any $x\in \mathbb{I}_n$, 
\[
\left\{\tau^Z_{k } \leq \lfloor bn^{\theta} \rfloor\right\} \;=\; \left\{\min_{0\leq i\leq \lfloor bn^{\theta} \rfloor} \|Z(i)\|_{\infty}\leq k \right\}\quad \text{under }\P_{x}.
\]

Therefore, for any initial state $x\in \mathbb{I}_n$, 
\begin{align}
   &  \P_{x} \left( \|X(t)\|_{d-1,\infty}\leq \frac{n}{2},\; N(t) < \lfloor bn^{\theta} \rfloor, \;  
   \tau^Z_{n/2\,+c_* } > \lfloor bn^{\theta} \rfloor
   \right) \notag\\
  =   &  \sum_{j=0}^{\lfloor bn^{\theta} \rfloor -1} \P_{x} \left(  \|X(t)\|_{d-1,\infty}\leq \frac{n}{2},\;\nu_j \leq t < \nu_{j+1}, \;  \min_{i\leq \lfloor bn^{\theta} \rfloor} \|Z(i)\|_{\infty} > \frac{n}{2} + c_*  \right) \notag\\
  \leq   &  \sum_{j=0}^{\lfloor bn^{\theta} \rfloor -1} \P_{x} \left(  \|X(t)\|_{d-1,\infty}\leq \frac{n}{2},\; \nu_j \leq t < \nu_{j+1}, \|Z(j)\|_{\infty} >  \frac{n}{2} +c_* , \; \|Z(j+1)\|_{\infty}> \frac{n}{2} +c_*  \right) \notag\\
  \leq & \sum_{j=0}^{\lfloor bn^{\theta} \rfloor -1} \P_{x} \left( \min_{\nu_j^D \leq m < \nu_{j+1}^D } \|X^D(m)\|_{d-1,\infty}\leq \frac{n}{2},\; \|Z(j)\|_{\infty} >  \frac{n}{2} +c_* ,\; \|Z(j+1)\|_{\infty}> \frac{n}{2} +c_* \right) \label{E:rare_excursion}
\end{align}

For each $0\leq j\leq \lfloor bn^{\theta} \rfloor -1$, our choice of the constant $c_*$ gives
\begin{align}
&\,\P_{x} \left( \min_{\nu_j^D \leq m < \nu_{j+1}^D } \|X^D(m)\|_{d-1,\infty}\leq \frac{n}{2},\;\|Z(j)\|_{\infty} >  \frac{n}{2} +c_* ,\; \|Z(j+1)\|_{\infty}> \frac{n}{2} +c_* \right)\\
=&\,\E_{x} \left[ 
  1_{\{ \|Z(j)\|_{\infty} >  \frac{n}{2}+c_* \}}\,\P_{(Z(j),0)}\left(
  \min_{0 \leq m < \nu_{1}^D } \|X^D(m)\|_{d-1,\infty}\leq \frac{n}{2},\;\|Z(1)\|_{\infty} >  \frac{n}{2} +c_*
   \right) \right]\\
\leq &\,
\E_{x} \left[ 
  1_{\{ \|Z(j)\|_{\infty} >  \frac{n}{2}+c_* \}}\,\P_{(Z(j),0)}\left( \left(\bigcup_{\eta \in \mathcal{T}^{(1)} \cup \mathcal{T}^{(2)}} E_{\eta}\right)^c\,
   \right) \right],
\end{align}
where  the equality follows from the strong Markov property of $X^D$, and the last inequality follows from  our choice of the constant $c_*$. By Assumption \ref{assump}(ii), for all $\ell > n/2+c_*> N_0$ and any initial state $y\in \mathbb{I}_{\ell}$, 
\begin{align*}
    \P_{y}\left( \left( \bigcup_{\eta \in \mathcal{T}^{(1)} \cup \mathcal{T}^{(2)}} E_{\eta}\right)^c \,\right) \leq \frac{c_2}{\ell^{\theta_2}} \leq \frac{c_2 2^{\theta_2}}{n^{\theta_2}}.
\end{align*} 
The  right of \eqref{E:rare_excursion} is therefore bounded above by 
\[
    \lfloor bn^{\theta} \rfloor\, c_2 2^{\theta_2} n^{-\theta_2} \leq\, b c_2 \cdot 2^{\theta_2} n^{\theta-\theta_2}
\]
for all $n>2(N_0-c_*)$. Note that the above argument still holds if $\nu_{K+1}=+\infty$ for some $K\in\mathbb{N}$. The proof of \eqref{eq:uppbdd_2} is complete.
\end{proof}


\subsection{Hitting estimate for the $d-1$-dimensional reduced chain}\label{sec: reduction}

In this section, we will give an upper bound for the second term on the right of \eqref{eq:thm1_also_0}, the probability $\P_{x}\left( \tau_{n/2\,+c_*}^Z \leq \lfloor bn^{\theta} \rfloor  \right)$. For the sake of simplicity, we prove the case of $c_*=0$. For $c_*>0$, the proof and the conclusion are the same up to some constant multiplication.

We can consider this as a problem solely about the $(d-1)$ dimensional process $Z$. 
Lemma \ref{L:1Dim_alt} is the only place where we need the ``laziness" condition \eqref{eq:reduction_assumption_Lazy} in Assumption \ref{assump}. We will  also need \eqref{eq:reduction_assumption} in Assumption \ref{assump}. 
Let 
$$c^* := \max_{j\in J^{(2)}} \|j\|_{\infty}= \max_{j\in J^{(2)}} \|j\|_{d-1,\infty},$$
where $J^{(2)}\subset \Z^{d-1}\backslash \{ {\bf 0}\}$ is the set of endpoint coordinates of paths in $\mathcal{T}^{(2)}$  defined in \eqref{Def:J2}.

\begin{lemma}\label{L:1Dim_alt}
Suppose Assumption \ref{assump}  holds. Then  for any $\theta,\,b\in (0,\infty)$, $n>2N_0$ and any initial state $x\in \mathbb{I}_n$, 
\begin{align}\label{eq:prob_term2_general}
\P_{x}\left( \tau_{n/2}^Z \leq \lfloor bn^{\theta} \rfloor \right) \leq \, b\, c_2  2^{\theta_2}\, n ^{\theta-\theta_2}\;+\;\frac{b\,n\,c_1\,2^{\theta_1}}{\sqrt{\pi \, \lfloor n/(2c^*) \rfloor}}\,\left( b\,3e\,c^*\,2^{\theta_1}\,c_1\,\cdot n^{\theta-(1+\theta_1)} \right)^{\lfloor n/(2c^*) \rfloor}.
    \end{align}
\end{lemma}

\begin{proof}[Proof of lemma \ref{L:1Dim_alt}] 
We split the event $\{\tau_{n/2}^{Z} \leq \lfloor bn^{\theta} \rfloor\}$ according to whether 
all increments $\{Z(i)-Z(i-1)\}_{i=1}^{\tau_{n/2}^{Z}}$ belong to $J^{(2)}\cup \{ {\bf 0}\}$, or not. 
For any $n>2N_0$ and $x\in \mathbb{I}_n$,
\begin{align}
	&\P_{x}\left( \tau_{n/2}^Z \leq \lfloor bn^{\theta} \rfloor,\; Z(i) - Z(i-1) \not \in J^{(2)} \cup \{ {\bf 0}\}  \text{ for some } i\leq \tau_{n/2}^Z \right) \notag\\   
 & \leq \sum_{i=1}^{\lfloor bn^{\theta} \rfloor}  \P_{x}\left( i\leq \tau_{n/2}^Z \leq \lfloor bn^{\theta} \rfloor, Z(i) - Z(i-1)\not \in J^{(2)} \cup \{ {\bf 0}\}  \right)  \notag \\
 &\leq \sum_{i=1}^{\lfloor bn^{\theta} \rfloor} \P_{x} \left( Z(i)- Z(i-1)\not \in J^{(2)}\cup \{ {\bf 0}\} , \|Z(i-1)\|_{\infty}\geq \frac{n}{2} \right)  \notag\\
    &\leq  bn^{\theta}\,\frac{c_2\cdot 2^{\theta_2}}{n^{\theta_2}} \;=\; b\, c_2  2^{\theta_2}\, n ^{\theta-\theta_2}.  \label{E:reduction_S1}
\end{align}
The last inequality follows because for each $1\leq i\leq \lfloor bn^{\theta} \rfloor$ and $n/2>N_0$, 
\begin{align*}
& \P_{x} \left( Z(i) - Z(i-1)\not \in J^{(2)}\cup \{ {\bf 0}\} , \|Z(i-1)\|_{\infty}\geq \frac{n}{2} \right)\\
=&\,\E_{x} \left[ 1_{\{ \|Z(i-1)\|_{\infty}\geq \frac{n}{2}\}}\, \P_{Z(i-1)} \left(Z(1) \not \in J^{(2)}\cup \{ {\bf 0}\}  \right)\right] \\
\leq &\,\E_{x} \left[ 1_{\{ \|Z(i-1)\|_{\infty}\geq \frac{n}{2}\}}\,\frac{c_2}{\Vert Z(i-1) \Vert^{\theta_2}_\infty}\right] \qquad\text{by }\eqref{eq:reduction_assumption} \text{ which come from  Assumption \ref{assump}(ii)}\\
\leq &\,\frac{c_2\cdot 2^{\theta_2}}{n^{\theta_2}}.
\end{align*}

For the other scenario where all the increments in the process $Z$ belong to $J^{(2)}\cup \{ {\bf 0}\}$, we note that
\begin{align}
&\P_{x}\left( \tau_{n/2}^{Z} \leq  \lfloor bn^{\theta} \rfloor, \; Z(i) - Z(i-1)  \in J^{(2)} \cup \{ {\bf 0}\}  \text{ for all } i\leq \tau_{n/2}^Z  \right) \label{E:1Dim_b}\\ 
=&\, \sum_{r=\lfloor  \frac{n}{2c^*} \rfloor+1}^{\lfloor bn^{\theta} \rfloor}\P_{x}\left( \tau_{n/2}^{Z} =r, \; Z(i) - Z(i-1)  \in J^{(2)} \cup \{ {\bf 0}\}  \text{ for all } i\leq \tau_{n/2}^Z   \right) \label{E:1Dim_b 2}
\end{align}
where $\lfloor  \frac{n}{2c^*} \rfloor+1$ in \eqref{E:1Dim_b 2} is the minimum number of jumps required to reach the set $\{z:\|z\|_{\infty}\le \frac{n}{2}\}$ from $x$, and 
we write
\[
K:=\left \lfloor  \frac{n}{2c^*} \right \rfloor.
\]
Under $\P_x$ where $x\in \mathbb{I}_n$,  if $\tau_{n/2}^{Z} =r$ occurs in this scenario, then there must be at least $K$ increments of $Z$ that belong to $J^{(2)}$ not $\{0\}$ during the first $r-1$ steps, and the increment given by the $r$ th jump must be in $J^{(2)}$.  Therefore, for each $r$ between $K+1$ and $\lfloor bn^{\theta} \rfloor$,
\begin{align*}
&\,\P_{x}\left( \tau_{n/2}^{Z} =r, \; Z(i) - Z(i-1) \in J^{(2)} \cup \{ {\bf 0}\}  \text{ for all } i\leq \tau_{n/2}^Z   \right) \notag\\
\leq &\, \sum_{\substack{(i_1,i_2,\cdots, i_{K})\\ \subset \{1,2,...,r-1\}}} \,\P _{x}\left( Z(i_1) - Z(i_1-1)\in J^{(2)}, \cdots, Z(i_{K}) - Z(i_{K}-1)\in J^{(2)},\; Z(r) - Z(r-1) \in J^{(2)} \right) \\
\leq &\, \binom{r-1}{K} \left(  \frac{c_1}{(n/2)^{\theta_1}} \right)^{K+1}, \notag
\end{align*}
where in the  last inequality, we used the fact that all  jumps with size in $J^{(2)}$ have probability at most $\frac{c_1}{(n/2)^{\theta_1}}$, by \eqref{eq:reduction_assumption_Lazy}(which comes from \eqref{eq:assump_1st} in Assumption \ref{assump}(i)). 
Let $c_1':=2^{\theta_1}c_1$ for simplicity. 
Since $r\le \lfloor bn^\theta \rfloor$, the last display is at most 
\begin{align*}
\binom{\lfloor bn^{\theta} \rfloor }{K} \left(  \frac{c_1'}{n^{\theta_1}} \right)^{K+1}
= &\, \frac{\lfloor bn^{\theta} \rfloor !}{K!\,(\lfloor bn^{\theta} \rfloor-K)!} \,\left(  \frac{c_1'}{n^{\theta_1}} \right)^{K+1}\\
\leq  &\, \frac{\lfloor bn^{\theta} \rfloor ^{K}}{K!} \,\left(  \frac{c_1'}{n^{\theta_1}} \right)^{K+1}\\
\leq  &\, \frac{\lfloor bn^{\theta} \rfloor ^{K}}{\sqrt{\pi K} \,(\frac{K}{e})^K}\,\left(  \frac{c_1'}{n^{\theta_1}} \right)^{K+1} \qquad\qquad \text{by Stirling's approximation} \\
= &\,\left(  \frac{c_1'}{n^{\theta_1}} \right)\,\frac{1}{\sqrt{\pi K}}\,\Bigg( \frac{\lfloor bn^{\theta} \rfloor}{\frac{K}{e}}\, \frac{c_1'}{n^{\theta_1}} \Bigg)^{K} \\
\leq &\,\left(  \frac{c_1'}{n^{\theta_1}} \right)\,\frac{1}{\sqrt{\pi K}}\,\left( 3\,b\,c^*\,e\,c_1'\cdot n^{\theta-(1+\theta_1)} \right)^{K}.
\end{align*}

Putting this and \eqref{E:reduction_S1} into \eqref{E:1Dim_b}, we obtain that
\[
\sup_{x\in \mathbb{I}_n}\P_{x}\left( \tau_{n/2}^Z \leq \lfloor bn^{\theta} \rfloor \right) \leq \, b\, c_2  2^{\theta_2}\, n ^{\theta-\theta_2}\;+\; \lfloor bn^{\theta} \rfloor\,
\left(  \frac{c_1'}{n^{\theta_1}} \right)\,\frac{1}{\sqrt{\pi K}}\,\left( 3\,b\,c^*\,e\,c_1'\cdot n^{\theta-(1+\theta_1)} \right)^{K}
\]
which tends to $0$ as $n\to\infty$, provided that either (i) $\theta<1+\theta_1$ or (ii)
$\theta=1+\theta_1$ and
$3\,b\,c^*\,e\,c_1' \in(0,1)$.

\end{proof}

\subsection{Putting everything together}

For any $\theta\in(0,\infty)$, we consider the following assumption.
\begin{assumption}\label{A:theta}
There exist constants $b_0$ and $C_0\in(0,\infty)$ such that
\begin{equation}\label{A:tauZ}
\limsup_{n\to \infty}\sup_{x\in \mathbb{I}_n}\P_{x}\left( \tau_{n/2+c_*}^Z \leq \lfloor bn^{\theta} \rfloor \right) \leq \,b\,C_0 \quad\text{ for all } b\in(0,b_0),   
\end{equation} 
where $c_*$ is the constant mentioned above \eqref{para:pathconstant}.
\end{assumption}

\begin{prop}\label{prop:nocycle}    
Let $X=\big(X(t)\big)_{t\in \R_{\ge 0}}$ be a continuous-time Markov process on $\Z_{\ge 0}^d$ that satisfies  
Assumption \ref{assump} and Assumption \ref{assump2}.  
Suppose Assumption \ref{A:theta}  holds for some $\theta\in(0,\theta_2]$. 
Then for any $\delta\in(0,1)$,     
    there exist  constants $N_{\delta}\in \mathbb{N}$ and $C_{\delta}\in (0,\infty)$ such that 
    for any $n\geq N_{\delta}$ and any initial state $x\in \mathbb{I}_n$, 
    \begin{equation}\label{eq:problowerbdd_2}
   \inf_{t\in[0,\,C_{\delta}\, n^{\theta}] } \P_{x} \left( \|X(t)\|_{d-1,\infty}> \frac{n}{2} \right) \geq \delta.
    \end{equation} 
\end{prop}

\begin{proof}
Upper bounds of the three terms on the right of \eqref{eq:thm1_also_0} and \eqref{eq:thm1_also_general} are given in 
Lemma \ref{lm:unit_Poisson_bdd} and  Lemma \ref{lm:rare_excursion_general} respectively.
From these upper bounds, for  any $t\in\R_{\ge 0}$, $\theta,\,b\in(0,\infty)$, $n> 4N_0$ and $x\in \mathbb{I}_n$, 
\begin{equation}\label{eq:thm1_also_2_general}
    \P_{x}\left(\|X(t)\|_{d-1,\infty}\leq \frac{n}{2}  \right) \leq 
    \frac{\kappa t}{\lfloor bn^{\theta} \rfloor}+ 
    \P_{x}\left( \tau^Z_{n/2\,+c_* } \leq \lfloor bn^{\theta} \rfloor  \right)
    + b c_2  2^{\theta_2}\,n^{\theta-\theta_2}.
\end{equation}
By  Assumption \ref{A:theta}, 
\begin{equation}\label{eq:thm1_also_2_general_2}
   \limsup_{n\to \infty}\sup_{x\in \mathbb{I}_n} \P_{x}\left(\|X(t)\|_{d-1,\infty}\leq \frac{n}{2}  \right) \leq 
   b \,C_0
    + b c_2  2^{\theta_2}.
\end{equation}

For any $\delta\in (0,1)$, 
we fix $b=b_{\delta}$ small enough so that $b \,C_0
    + b c_2  2^{\theta_2}<\frac{1-\delta}{2}$. Hence there exists $N_{\delta}$ such that 
    for any $n\geq N_{\delta}$, we have
\[
\sup_{x\in \mathbb{I}_n} \P_{x}\left(\|X(t)\|_{d-1,\infty}\leq \frac{n}{2}  \right) <1-\delta.
\]

We have completed the proof of Proposition \ref{prop:nocycle}.
\end{proof}


\begin{prop}\label{prop:cycle gives tauZ}    
Let $X=\big(X(t)\big)_{t\in \R_{\ge 0}}$ be a continuous-time Markov process on $\Z_{\ge 0}^d$ that satisfies  
Assumption \ref{assump}. Then 
Assumption \ref{A:theta} holds with $\theta:= \min\{1+\theta_1,\,\theta_2\}$.
\end{prop}

\begin{proof}
Proposition \ref{prop:cycle gives tauZ} follows immediately from Lemma \ref{L:1Dim_alt}, by taking $b_0=\frac{1}{3e |c^*|\,2^{\theta_1}\,c_1}$ and $C_0=c_2 2^{\theta_2}$. More precisely, 
$\theta\leq \theta_2$ implies that the first term on the right of \eqref{eq:prob_term2_general} is at most $b\,c_2 2^{\theta_2}$. Furthermore, 
 $\theta\leq 1+\theta_1$ implies that the second term on the right of \eqref{eq:prob_term2_general} is zero after taking  $\limsup_{n\to\infty}$.
\end{proof}

Now we turn to the proof of Theorem \ref{thm}.
\begin{proof}[Proof of Theorem \ref{thm}]
Theorem \ref{thm} follows immediately from Propositions \ref{prop:nocycle} and \ref{prop:cycle gives tauZ}. This is because Proposition \ref{prop:cycle gives tauZ} implies that Assumption \ref{A:theta} holds with $\theta:= \min\{1+\theta_1,\,\theta_2\}$. Then  Proposition \ref{prop:nocycle} implies that
\eqref{eq:problowerbdd_2} holds with $\theta:= \min\{1+\theta_1,\,\theta_2\}$. The latter is the desired inequality \eqref{eq:thm1_also}.
\end{proof}



\section{Proofs of Theorems \ref{prop:cycle} and \ref{prop:cycle2}} \label{sec:proof 2}

In the discussion below, we mainly consider the dynamic of stochastic model when species $A$ is prevalent and species $B$ is absent, which is consistent with our choice of initial conditions in Assumption \ref{assump}. And throughout this section, we will write $f(x)= O(g(x))$ if 
$$\limsup_{x_1\to \infty}\frac{f(x)}{g(x)}\in (0,\infty).$$
where $x_1$ is the first coordinate of $x$, which represents the number of $A$ species.

\begin{proof}[Proof of Theorem \ref{prop:cycle}]

Recall from \eqref{eq:dominantcycle} that for any sequence  $\eta=(\eta_k)_{k=1}$ of elements in $\Z^d$, the event that $X^D$ follows $\eta$ is 
\begin{align}
           E_{\eta}:=\{X^D(k)-X^D(k-1)=\eta_k \quad \text{for } k=1,2,\dots,|\eta|\}.
\end{align}
where $|\eta|$ is the length of the sequence. For any sequence $\eta\in \Z^d$ and any $i=1,2,3,\dots,|\eta|$, we define
\begin{align}\label{eq:E not i}
           E^{\neq i}_{\eta}:=\{X^D(k)-X^D(k-1)=\eta_k \quad \text{for } k=1,2,\dots,i-1,\quad \text{and }X^D(i)- X^D(i-1)\neq \eta_i\}
\end{align}
be the event that $X^D$ exits the path  $X^D(0)+\gamma_\eta$ at the $i$ th transition. 

With these notations, we will show Assumption \ref{assump}(i) holds with
\begin{align}\label{eq:T1_in_assump}
    \theta_1 = \min_{i} \{\alpha_{i-1} - \alpha_{i-2}: 2\leq i \leq L\}, \qquad \mathcal{T}^{(1)} = \{ \eta^0 \},
\end{align}
where  $\eta^0$ is the sequence of vectors specified by the reaction vectors, i.e.
\begin{align} \label{eq:T1traj} 
    \eta^0 & := \left(z_1 - z_0, z_2-z_1, \cdots, z_{L-1} - z_{L-2}, z_0 - z_{L-1}\right).
\end{align}
Assumption \ref{assump}(i) requires that the path associated with $\eta^0$ is a cycle. The latter clearly holds since  
\[
    \sum_{i=1}^L \eta_i^0 = \sum_{i=1}^L (z_{i} -z_{i-1})  = 0.
\]
To verify Assumptions \ref{assump}(i), we will focus on the complement of the event, 
\begin{align*}
    \P_{(n,0)}\left( E_{\eta^0}^c \right ) = \sum_{i=1}^L    \P_{(n,0)}\left( E^{\neq i}_{\eta^0} \right). 
\end{align*}
Due to absence of B species at $(n,0)$,  $z_0\rightarrow z_1$ is the only possible jump and $\P_{(n,0)}\left( E_{\eta^0}^{\neq 1} \right ) = 0$. For $2\leq i \leq L$, by \eqref{eq:transition prob of XD}, 

{\small
\begin{align*}
    &\P_{(n,0)}\left( E^{\neq i}_{\eta^0} \right ) \\ 
     = &\P_{(n,0)}\Big(X^D(1) - X^D(0) = z_1-z_0,\,
     X^D(2) - X^D(1) = z_2-z_1, \;\ldots  \\
    & \qquad \qquad \ldots,\, X^D(i-1) - X^D(i-2)  = z_{i-1}- z_{i-2}, \, X^D(i) - X^D(i-1)\not = z_i  -z_{i-1} \Big)  \\
    = &  \prod_{k=1}^{i-1} \P_{(n,0)+z_{k-1}}\left( X^D(1) - X^D(0) = z_{k} -z_{k-1}  \right ) \cdot   \P_{(n,0)+z_{i-1}}\left(  X^D(1) - X^D(0) \not = z_i-z_{i-1}  \right)  
\end{align*}
\begin{equation}\label{eq:asdasihfasd}
    = \prod_{k=1}^{i-1}  \frac{\displaystyle \sum_{ \substack{\ell: \beta_{\ell} \leq \beta_{k-1} \\ z_{\ell+1}- z_{\ell} = z_k - z_{k-1}}}\kappa_{\ell} \frac{(n+\alpha_{k-1})!(\beta_{k-1})!}{(n+\alpha_{k-1}-\alpha_\ell)!(\beta_{k-1} - \beta_\ell)!}}{\displaystyle \sum_{\ell: \beta_\ell \leq \beta_{k-1}} \kappa_{\ell}\frac{(n+\alpha_{k-1})!(\beta_{k-1})!}{(n+\alpha_{k-1}-\alpha_{\ell})!(\beta_{k-1}-\beta_{\ell})!}} \cdot \frac{\displaystyle\sum_{ \substack{\ell: \beta_\ell \leq \beta_{i-1} \\ z_{\ell+1}- z_{\ell} \not= z_i - z_{i-1}}} \kappa_{\ell} \frac{(n+\alpha_{i-1})!(\beta_{i-1})!}{(n+\alpha_{i-1}-\alpha_{\ell})!(\beta_{i-1}-\beta_{\ell})!}}{\displaystyle\sum_{\ell: \beta_\ell \leq \beta_{i-1}} \kappa_{\ell} \frac{(n+\alpha_{i-1})!(\beta_{i-1})!}{(n+\alpha_{i-1}-\alpha_{\ell})!(\beta_{i-1}-\beta_{\ell})!}}
\end{equation}}
Reaction rates in the direction of $z_k-z_{k-1}$ is proportional to the total reaction intensity sharing the same direction, namely $\displaystyle \sum_{ \substack{\ell: z_{\ell+1}- z_{\ell} = z_k - z_{k-1}}} \lambda_\ell(x)$. 
However due to the choice of mass-action kinetics, $\lambda_{\ell}(x) = 0$ if $\beta_\ell > x_2$, and hence 
\[
    \sum_{\substack{\ell: z_{\ell+1}- z_{\ell} = z_k - z_{k-1}}}  \lambda_\ell (x) = 
    \sum_{ \substack{\ell: \beta_\ell \leq x_2 \\ z_{\ell+1}- z_{\ell} = z_k - z_{k-1}}} \kappa_{\ell} \frac{x_1!x_2!}{(x_1-\alpha_{\ell})!(x_2-\beta_{\ell})!}
\]
Similarly, reaction rates for all possible transitions are given by 
\[
    \sum_{\substack{\ell}}  \lambda_\ell (x) = 
    \sum_{ \substack{\ell: \beta_\ell \leq x_2}} \kappa_{\ell} \frac{x_1!x_2!}{(x_1-\alpha_{\ell})!(x_2-\beta_{\ell})!.}
\] Using these transition rates, we can derive \eqref{eq:asdasihfasd} combined with \eqref{eq:transition prob of XD}.

Now, we investigate the order of magnitude of each term shown in \eqref{eq:asdasihfasd}. Since $\beta_i$ is strictly increasing, $\{\beta_\ell\le \beta_{k-1}\}$ and $\{\beta_\ell \le \beta_{i-1}\}$ can be replaced by $\{\ell\le k-1\}$ and $\{\ell \le i-1\}$, respectively in \eqref{eq:asdasihfasd}. Thus
{\small
\begin{align}\label{eq:betaincreasing}
    \sum_{\ell: \beta_\ell \leq \beta_{k-1}} \kappa_{\ell}\frac{(n+\alpha_{k-1})!(\beta_{k-1})!}{(n+\alpha_{k-1}-\alpha_{\ell})!(\beta_{k-1}-\beta_{\ell})!} = \sum_{\ell=0}^{k-1} \kappa_{\ell} \frac{(n+\alpha_{k-1})!(\beta_{k-1})!}{(n+\alpha_{k-1}-\alpha_{\ell})!(\beta_{k-1}-\beta_{\ell})!}.
\end{align}}
Since $\alpha_i$ is strictly increasing, 
{\small
\begin{align}\label{eq:alphaincreasing}
    \sum_{\ell=0}^{k-1} \kappa_{\ell} \frac{(n+\alpha_{k-1})!(\beta_{k-1})!}{(n+\alpha_{k-1}-\alpha_{\ell})!(\beta_{k-1}-\beta_{\ell})!}= O\left(\kappa_{k-1}\frac{(n+\alpha_{k-1})!(\beta_{k-1})!}{n!} \right)= O\left(\frac{(n+\alpha_{k-1})!}{n!} \right),
\end{align}}
Similarly by monotonicity of $\beta_i$ and $\alpha_i$, 
{\small
\begin{align}\label{eq:asdasfafasdasd}
\begin{split}
    \sum_{ \substack{\ell: \beta_\ell \leq \beta_{k-1} \\ z_{\ell+1}- z_{\ell} =  z_k - z_{k-1}}}\kappa_{\ell} \frac{(n+\alpha_{k-1})!(\beta_{k-1})!}{(n+\alpha_{k-1}-\alpha_\ell)!(\beta_{k-1} - \beta_\ell)!} & = \sum_{\substack{\ell\leq k-1:\\ z_{\ell+1}- z_{\ell} =  z_k - z_{k-1}}} \kappa_{\ell} \frac{(n+\alpha_{k-1})!(\beta_{k-1})!}{(n+\alpha_{k-1}-\alpha_\ell)!(\beta_{k-1} - \beta_\ell)!}  \\ 
    & = O\left(\frac{(n+\alpha_{k-1})!}{n!} \right) \\ 
    \sum_{ \substack{\ell: \beta_\ell \leq \beta_{i-1} \\ z_{\ell+1}- z_{\ell} \not= z_i - z_{i-1}}} \kappa_{\ell} \frac{(n+\alpha_{i-1})!(\beta_{i-1})!}{(n+\alpha_{i-1}-\alpha_{\ell})!(\beta_{i-1}-\beta_{\ell})!} & = \sum_{ \substack{\ell \leq i-1:  \\ z_{\ell+1}- z_{\ell} \not= z_i - z_{i-1}}} \kappa_{\ell} \frac{(n+\alpha_{i-1})!(\beta_{i-1})!}{(n+\alpha_{i-1}-\alpha_{\ell})!(\beta_{i-1}-\beta_{\ell})!}\\
    & \leq O \left(\frac{(n+\alpha_{i-1})!}{(n+\alpha_{i-1}-\alpha_{i-2})!}\right)
\end{split}
\end{align}}
Combining equation \eqref{eq:alphaincreasing} and \eqref{eq:asdasfafasdasd} in \eqref{eq:asdasihfasd}, we have
\begin{align}
      \P_{(n,0)}\left( E^{\neq i}_{\eta^0} \right )
     =  \prod_{k=1}^{i-1}  O  \left(\frac{\displaystyle \frac{(n+\alpha_{k-1})!}{n!} }{\displaystyle\frac{(n+\alpha_{k-1})!}{n!} } \right)   O\left( \frac{\displaystyle \frac{(n+\alpha_{i-1})!}{(n+\alpha_{i-1}-\alpha_{i-2})!}}{\displaystyle\frac{(n+\alpha_{i-1})!}{n!}} \right) &=  O \left (\frac{n!}{(n+\alpha_{i-1}-\alpha_{i-2})!}\right) \notag \\
    &= O\left( n^{\alpha_{i-2} - \alpha_{i-1}} \right). \label{eq:sadajahsfuauis} 
 \end{align}
Consequently, since $\alpha_i$ is increasing, 
\begin{align*}
       \P_{(n,0)}\left( E_{\eta^0}^c \right ) = \sum_{i=1}^L    \P_{(n,0)}\left( E^{\neq i}_{\eta^0} \right ) = \sum_{i=2}^L  O\left( n^{\alpha_{i-2} - \alpha_{i-1}} \right) = O\left(n^{-\theta_1}\right),
\end{align*}
where $\theta_1 = \min_{i} \{\alpha_{i-1} - \alpha_{i-2}: 2\leq i \leq L\}$ specifies the probability, as in Assumption \ref{assump}(i), of most likely excursions from the $\eta^0$ \eqref{eq:T1traj}. Equation \eqref{eq:assump_1st} is then verified by considering the complement event.

Finally, Assumption \ref{assump2} holds with $\kappa = \kappa_0$. This is because   any stochastic trajectory started on the axis $\{x_2 = 0\}$ only leave the axis after the occurrence of reaction $\emptyset \rightarrow \alpha_1A +\beta_1 B$, which takes an exponential time with parameter $\kappa_0$. The proof of Theorem \ref{prop:cycle} is now complete.
\end{proof}

\medskip

Throughout the proof of Theorem \ref{prop:cycle2}, we investigate the probability of each transition along the trajectory of $X^D(t)$. At any state $x\in \Z^2_{\ge 0}$ with $1\leq x_2\leq L$ and large enough $x_1$, probability of each transition is proportional to the intensity of reaction $z_{\ell-1} \rightarrow z_\ell$ in \eqref{mass}, 
\begin{equation}\label{eq:massinproposition}
    \lambda_{z_{\ell-1} \rightarrow z_\ell}(x) = \kappa_\ell \frac{x_1! }{(x_1-\alpha_{\ell-1})} \frac{x_2!}{(x_2-\beta_{\ell-1})} = O(x_1^{\alpha_{\ell-1}}) 1_{\beta_{\ell-1} \leq x_2}. 
\end{equation}
The \textit{most likely reaction} is the reaction $z_{i-1}\rightarrow z_{i}$, or the ordered pair $(z_{i-1},z_i)\in \mathcal{R}$, that maximizes the intensity function $\lambda_{z_{\ell-1} \rightarrow z_\ell}(x)$ in \eqref{eq:massinproposition} among all $\ell$, and 
by Assumption \ref{assump:cyclic}, 
\begin{equation}\label{eq: mostlikelyreaction}
    i = \argmax \{\alpha_{\ell-1}: \beta_{\ell-1} \leq x_2\} = \max \{\ell: \beta_{\ell-1} \leq x_2\} = x_2 +1 
\end{equation}
The \textit{second most likely reaction} is the reaction $z_{j-1}\rightarrow z_{j}$, or the ordered pair $(z_{j-1},z_{j})\in \mathcal{R}$, that achieves the second largest intensity function $\lambda_{z_{\ell-1} \rightarrow z_\ell}(x)$ in \eqref{eq:massinproposition} among all $\ell$, or more specifically,
\[
    j = \argmax \{\alpha_{\ell-1}: \beta_{\ell-1} \leq x_2, \ell\neq i\} = \max \{\ell: \beta_{\ell-1} \leq x_2\} - 1 = x_2,
\]
where $i$ is the index in \eqref{eq: mostlikelyreaction} for the \textit{most likely reaction}. 

\medskip

\noindent
\begin{proof}[Proof of Theorem \ref{prop:cycle2}]

Note that under Assumption \ref{assump:cyclic}, both $\alpha_i$ and $\beta_i$ are increasing. Hence by Theorem \ref{prop:cycle},  Assumption \ref{assump}(i) is satisfied with 
\begin{align*}
    \theta_1 = \min_{i} \{\alpha_{i-1} - \alpha_{i-2}: 2\leq i \leq L\}, \qquad \mathcal{T}^{(1)} = \{ \eta^0 \},
\end{align*}
and Assumption \ref{assump2} is satisfied with $\kappa = \kappa_0$.

We next show that Assumption \ref{assump}(ii) holds for
\begin{equation}\label{eq:theta_in_assump}
    \theta_2  = \min_{2\leq i\leq L: \left(\alpha_{i-1} - \alpha_{i-2}\right)>\theta_1} \left(\alpha_{i-1} - \alpha_{i-2}\right) \wedge 2\theta_1 > \theta_1,
\end{equation}
and
    \begin{align}\label{eq:T2_in_assump}
    \mathcal{T}^{(2)} = \{\eta^i: \alpha_{i-1} - \alpha_{i-2} = \theta_1, 2\leq i\leq L \}\neq \emptyset,
\end{align}
where, for all $i = 2,3, \cdots, L-1$, $\eta^i$ is a sequence with length $L$, defined by
\begin{align}\label{eq:T2traj} 
    \eta^i & =  \left(\eta_1^i = z_1 - z_0,, \cdots, \eta_{i-1}^i = z_{i-1} - z_{i-2},\;  \eta_{i}^i=\eta_{i-1}^i, \;\eta_{i+1}^i = z_{i+1} - z_{i} \cdots ,\eta_{L}^i = z_0 - z_{L-1}\right), 
\end{align}
and $\eta^L$ is a sequence with length $2L$, defined by
\begin{align}\label{eq:T2trajTop} 
    \eta^L & =  ( \eta_1^L = z_1 - z_0,    \eta_2^L = z_2-z_1,\cdots, \eta_{L-1}^L  = z_{L-1} - z_{L-2},   \eta_{L}^L  = \eta_{L-1}^L, \notag \\
    & \eta_{L+1}^L = z_0 - z_{L-1},  \eta_{L+2}^L = z_2 - z_1, \cdots,  \eta_{2L-1}^L = z_{L-1} - z_{L-2}, \eta_{2L} = Z_0-z_{L-1} ). 
\end{align}

To identify the sequences of transitions in $\mathcal{T}^{(2)}$, we start by following the cyclic trajectory $\eta^0$ \eqref{eq:T1traj} until exiting at the $i$ th transition with the reaction having the second highest probability.  Note that Assumption \ref{assump}(ii) requires that the path associated with $\eta^i$ ends on the axis $\{x_2=0\}\setminus \{{\bf 0}\}$ for $2\leq i \leq L$.

\begin{enumerate}
    \item For $i = 2,3 ,... , L-1$, $\eta^i$ exits the cyclic trajectory $\eta^0$ \eqref{eq:T1traj} at the $i$ th transition. More specifically $\eta^i$ shares the same transitions as $\eta^0$ for the first $i-1$ steps. Upon exiting, the $i$ th transition is given by the \textit{second most likely reaction} $z_{i-2}\rightarrow z_{i-1}$ instead of the \textit{most likely reaction} $z_{i-1} \rightarrow z_i$ as in  $\eta^0$. Note that for both $\eta^i$ and $\eta^0$, the number of species B are both given by $i$ after $i$ th transition (the assumption  $\beta_j - \beta_{j-1}=1$ for all $j$ is used here). Hence, all subsequent \textit{most likely reactions} for $\eta^i$ remain the same as $\eta^0$ until returning to the $x-$axis. 
    
    By following \textit{most likely reactions} except the $i$ th transition, which follows \textit{the second most likely reaction}, $\eta^i$ should be the trajectory exiting $\eta^0$ at $i$ th transition with the highest probability. As the path associated with $\eta^i$ ends on the axis $\{x_2=0\}\setminus \{{\bf 0}\}$, the accumulated changes in A species for $\eta^i$ can be computed by
    \begin{align*}
        x_A^{\eta^i} & := \left(\sum_{j=1}^L \eta_j^i \right)_A \\
        &= \left( z_1 - z_0 + \cdots + z_{i-1} -z_{i-2} +   \underbrace{(z_{i-1} -z_{i-2})}_{\text{2nd most likely reaction}} +z_{i+1} - z_{i}+ \cdots  + z_0 -z_{L-1}\right)_A  \\
        & = \left( z_{i-1} - z_0 + z_{i-1} - z_{i-2} +  z_ 0 - z_i \right)_A = \left(  2z_{i-1} - z_{i-2} - z_i \right)_A  = 2\alpha_{i-1} - \alpha_{i-2}- \alpha_i. 
    \end{align*}

    \item For $i=L$, $\eta^L$ exits the cyclic trajectory $\eta^0$ \eqref{eq:T1traj}  at the $L$ th transition. More specifically $\eta^L$ share the same sequence of transitions for the first $L-1$ transitions. Upon exiting, the $L$ th transition is given by the \textit{second most likely reaction} $z_{L-2}\rightarrow z_{L-1}$ instead of the \textit{most likely reaction} $z_{L-1} \rightarrow z_0$ as in $\eta^0$. The number of B species after $L$ th transition is $L$ (the assumption  $\beta_j - \beta_{j-1}=1$ for all $j$ is used here). Hence the most likely reaction $z_{L-1} \rightarrow z_0$ occurs, which reduce the number of B species to 1. Then all subsequent \textit{most likely reactions} share the same trajectory as $\eta^0$ \eqref{eq:T1traj} after its first transition (the assumption  $\beta_j=j$ for all $j$ is used here), namely $\eta^L_{L+\ell} = \eta^0_\ell$ for any $2\leq \ell \leq L$. 

    In summary, by following most likely reactions except the $L$ th transition, which follows the \textit{second most likely reaction}, $\eta^L$ should be the trajectory exiting $\eta^0$ at $L$ th transition with the highest probability. As the path associated with $\eta^L$ ends on the axis $\{x_2=0\}\setminus \{{\bf 0}\}$, the accumulated changes in A species for $\eta^L$ can be computed by
    \begin{align*}
        x_A^{\eta^L} := \left(\sum_{j=1}^{2L} \eta_j^L \right)_A & = (  z_1 - z_0 + \cdots + z_{L-1} -z_{L-2} + \underbrace{(z_{L-1} -z_{L-2})}_{\text{2nd most likely reaction}} + z_0 - z_{L-1} \\ 
         & +   z_2-z_1 + \cdots + z_{L-1} -z_{L-2}  + z_0 -z_{L-1} )_A  \\
        & = \left( z_{L-1} - z_0 + z_{L-1} - z_{L-2} +  z_ 0 - z_1 \right)_A = \left( z_{L-1}-z_{L-2} +z_0 -z_1\right)_A \\ 
        &  = \alpha_{L-1} - \alpha_{L-2}- \alpha_1. 
    \end{align*}  
\end{enumerate}

Next we verify equation \eqref{eq:assump_3rd} with $\theta_2$ in \eqref{eq:theta_in_assump}. We will follow the similar approach in the proof of Theorem \ref{prop:cycle} by considering the complement of the event $\displaystyle \bigcup_{\eta \in \mathcal{T}^{(1)} \cup \mathcal{T}^{(2)}} E_{\eta}$, which happens in one the following three scenarios. 

\begin{enumerate}
    \item The stochastic process follows the cyclic trajectory $(n,0)+\gamma_{\eta^0}$ until exiting at $i$ th transition where $\eta^i\not \in \mathcal{T}^{(2)}$, whose probability is computed in \eqref{eq:sadajahsfuauis}, and can be bounded by 
    \begin{align*}
         \sum_{j: \eta^j \not \in \mathcal{T}^{(2)}} \P_{(n,0)}\left( E^{\neq j}_{\eta^0}  \right) \leq  \sum_{j: \eta^j\not \in \mathcal{T}^{(2)}} O\left(n^{\alpha_{j-2} - \alpha_{j-1}}\right) = O\left( n^{-\tilde \theta}\right), 
    \end{align*} 
    where $\displaystyle \tilde\theta = \min_{2\leq i\leq L:\, \left(\alpha_{i-1} - \alpha_{i-2}\right)>\theta_1} \left(\alpha_{i-1} - \alpha_{i-2}\right)$. 
    
    \item The stochastic process follows the cyclic trajectory $(n,0)+\gamma_{\eta^0}$ until exiting at $i$ th transition where $\eta^i \in \mathcal{T}^{(2)}$, however the transition is neither the most likely, nor the second most likely. Then, 
    \begin{align*}
        & \P_{(n,0)}\left( X^D(1) - X^D(0)  = \eta_{1}^i, \cdots, X^D(i-1) - X^D_{i-2} =\eta_{i-1}^i  , X^D(i) - X^D(i-1) \not= z_i - z_{i-1} \text{ or } z_{i-1} - z_{i-2},\right) \\
        = & \prod_{k=1}^{i-1} \P_{(n,0)+z_{k-1}}\left( X^D(1) - X^D(0) = z_{k} -z_{k-1}  \right )  \P_{(n,0)+z_{i-1}}\left(  X^D(1) - X^D(0) \not= z_i - z_{i-1} \text{ or } z_{i-1} - z_{i-2}, \right) \\
        \leq & \P_{(n,0)+z_{i-1}}\left(  X^D(1) - X^D(0) \not= z_i - z_{i-1} \text{ or } z_{i-1} - z_{i-2} \right) \\
        \leq & \frac{ \displaystyle \sum_{\ell=0}^{i-3} \kappa_{\ell} \frac{(n+\alpha_{i-1})!(\beta_{i-1})!}{(n+\alpha_{i-1}-\alpha_{\ell})!(\beta_{i-1}-\beta_{\ell})!} }{\displaystyle \sum_{\ell=0}^{i-1} \kappa_{\ell} \frac{(n+\alpha_{i-1})!(\beta_{i-1})!}{(n+\alpha_{i-1}-\alpha_{\ell})!(\beta_{i-1}-\beta_{\ell})!} } = O\left(  \dfrac{ \kappa_{i-3}\dfrac{(n+\alpha_{i-1})!(\beta_{i-1})!}{(n+\alpha_{i-1}-\alpha_{i-3})!(\beta_{i-1}-\beta_{i-3})!}}{\kappa_{i-1}\dfrac{(n+\alpha_{i-1})!(\beta_{i-1})!}{n!}} \right) = O\left(n^{\alpha_{i-3} - \alpha_{i-1}} \right)
    \end{align*}

    \item The stochastic process follows the cyclic trajectory $(n,0)+\gamma_{\eta^0}$ until exiting at $i$ th transition where $\eta^i \in \mathcal{T}^{(2)}$, however the trajectory does not follow $\eta^i$ and it exits $\eta^i$ at some $k$ th transition for some $i<k \leq |\eta^i|$. Then this event can be represented as $E^{\neq k}_{\eta^i}$, and 

    \begin{align*}
        & \P_{(n,0)}\left( E^{\neq k}_{\eta^i} \right) = \P_{(n,0)}\left(X^D(1) - X^D(0) = \eta^i_1,\cdots, X^D(k-1) - X^D_{k-2} = \eta^i_{k-1}, X^D(k) - X^D(k-1) \neq \eta^i_k  \right) \\
        \leq & \P_{(n,0)}\left( E^{\neq i}_{\eta^0} \right) \P_{X^D(k-1)}\left( X^D(1) - X^D(0) \neq \eta^i_k \right) =  O\left( n^{-\theta_1}\right) \cdot O\left( n^{-\theta_1}\right)  
    \end{align*}
    Note that at the last equation $\P_{X^D(k-1)}\left( X^D(1) - X^D(0) \neq \eta^i_k \right) = O\left( n^{-\theta_1}\right)$ holds since all transitions except $i$ th transition are most likely reactions, which follows similar calculations in $ \P_{(n,0)}\left( E^{\neq i}_{\eta^0} \right)$ established in the proof of Theorem \ref{prop:cycle}. Hence, 
    \begin{align*}
        & \P_{(n,0)}\left( \bigcup_{i<k\leq |\eta^i|} E^{\neq k}_{\eta^i} \right)  \leq \sum_{k=i+1}^{|\eta^|} \P_{(n,0)}\left( E^{\neq k}_{\eta^i} \right) = \sum_{k=i+1}^{|\eta^i|}  O\left( n^{-\theta_1}\right) \cdot O\left( n^{-\theta_1}\right)  = O\left( n^{-2\theta_1}\right)
    \end{align*} 
\end{enumerate}

In summary, 
\begin{align*}
    \P_{(n,0)}\left( \bigcup_{\eta \in \mathcal{T}^{(1)} \cup \mathcal{T}^{(2)}} E_{\eta} \right ) & = 1 - \P_{(n,0)}\left( \left(\bigcup_{\eta \in \mathcal{T}^{(1)} \cup \mathcal{T}^{(2)}} E_{\eta} \right)^c \right) \\ 
    \geq &\, 1 -\sum_{j: \eta_j\not \in \mathcal{T}^{(2)}} O\left(n^{\alpha_{j-2} - \alpha_{j-1}}\right) - O\left(n^{\alpha_{j-3} - \alpha_{j-1}} \right) - O\left( n^{-2\theta_1}\right)   \geq 1- \frac{c_2}{n^{\theta_2}}, 
\end{align*}
where $\displaystyle \theta_2  = \min_{2\leq i\leq L: \left(\alpha_{i-1} - \alpha_{i-2}\right)>\theta_1} \left(\alpha_{i-1} - \alpha_{i-2}\right) \wedge 2\theta_1$. 

The proof of Theorem \ref{prop:cycle2} is complete.
\end{proof}

\section{Proof of the upper bound in Proposition \ref{prop:mean first hitting}} \label{sec:proof 3}

\begin{proof}[Proof of Proposition \ref{prop:mean first hitting}]
Let $X$ be the associated Markov chain for a reaction network $(\Sp,\C,\Re)$ given in \eqref{eq:cyclicmodel} with Assumption \ref{assump:cyclic}.
Then as \eqref{Def:Zi}, we use the same discrete-time 1-dimensional Markov chain 
\begin{align*}
Z(i):=
X_A(\nu_i)=X_A^D(\nu_i^D)  
\end{align*}
where 
$\nu_i$ are the same as \eqref{eq:StoppingDef}.
We will show that for $\tau^Z_{C}:=\inf \{i > 0: Z(i) \le C\}$ as in \eqref{Def:tauZ}
, we have $\mathbb E(\tau^Z_{C} | Z(0)=n) \le c n^\theta$ for some $c>0$. If this holds, \eqref{eq:mean first passage time} follows because of the following reasons: Let $X^D$ be the embedded Markov chain of $X$. 
\begin{enumerate}
    \item[i)] Note that under mass-action \eqref{mass}
    \begin{align*}
        \sup_{x\in \mathbb  Z^d_{\ge 0}} \frac{1}{\mathbb E \left (\sum_{i=0}^{L-1}\lambda_i(x) \right )} \le \frac{1}{k_0},
    \end{align*}
    where $\lambda_i$ is the intensity function of reaction $z_i\to z_{i+1}$.
    Therefore the average time of each transition of $X$ (i.e. exponential holding time) has a uniform upper bound for each state $x$. However, the transition of $X^D$ takes a unit time. Hence, up to a constant, the total transition time of $X$ to $\{x_1 \in\Z_{\ge 0} : x_1\le C\}$ takes a shorter time than that of $X^D$ for any $C>0$.   
    \item[ii)] To show i) more precisely, we let 
    \begin{align*}
     \tau^D_C=\inf\{i: \Vert X^D(i) \Vert_{2-1,\infty}=X^D_A(i)\le C\} \quad \text{for each $C\ge 0$}.    \end{align*}
    We also let $\Gamma$ be the collection of all paths starting from $(n,0)$ and arriving at some $x \in \{ x: \Vert x \Vert_{2-1,\infty}\le C\}$. Then letting $\mathbb E^D$ and $\mathbb E$ be the expectation under $X^D$ and $X$, respectively, we have
    \begin{align*}
        \E_{(n,0)}[\tau_C] &= \sum_{\eta\in \Gamma}\E_{(n,0)}[\tau_C|E_\eta]\,\P_{(n,0)}(E_\eta)
        \le C' \sum_{\eta\in \Gamma} \mathbb E^D_{(n,0)}[\tau^D_C | E_\eta ]\,\P_{(n,0)}(E_\eta)=C'\mathbb E^D[\tau^D_C],
    \end{align*}
    where we used the inequality $\E[\tau_C|E_\eta] \le C' \mathbb E[\tau^D_C|E_\eta]$ that comes from the fact that each transition of $X$ takes shorter time than the same transition of $X^D$ up to the constant $C'$.
    \item[iii)] By the definitions of $X^D$ and $Z$, we have that $\tau^{D}_{C} \leq \tau^Z_{C}$ almost surely under $\P_{(n,0)}$, whenever $0<C<n$. In other words, $X^D$ always hits $\{x \in \mathbb Z^2_{\ge 0} : \Vert x \Vert_{2-1,\infty} < C\}$ before $Z$ hits the set $\{x_1  \in\Z_{\ge 0}: x_1\le C\}$.
\end{enumerate}
Therefore, it suffices to show that there exits a constant  $c>0$ such that $\E[\tau^Z_{C} | Z(0)=n] \le c n^{\theta}$ for all large enough $n$.
To show this, we will show existence of constants $c_0,\,N_0\in(0,\infty)$ such that 
\begin{equation}\label{E:Lyapunov}
A^Z V(n)\le -c_0  \qquad \text{for }n\geq N_0,    
\end{equation}
where   $V(n)=n^{\theta}=n^{\theta_1+1}$ for $n\in\Z_{\ge 0}$ and $\mathcal A^Z$ is the generator of $Z$ such that $A^ZV(n)=\mathbb E_n(V(Z(1)))-V(n)$. 


To establish \eqref{E:Lyapunov}, in Lemma \ref{lem:transition prob of Z} we derived some estimates of the transition probabilities of $Z$, which give more precise bounds than those in Lemma \ref{L:2Dto1D}.
By Lemma \ref{lem:transition prob of Z}, there exists $\widetilde{N}\in\mathbb{N}$ such that
\begin{align}
\begin{split}\label{eq:increment of Z}
    & C_{*}:=\inf_{n\geq \widetilde{N}} n^{\theta_1}\E_{n}[(Z(0)-Z(1))\,1_{\{Z(1)-Z(0)<0\}}] >0, \\
    & \sup_{n\geq \widetilde{N}}\E_{n}[(Z(1)^{\theta_1+1}-Z(0)^{\theta_1+1})\,1_{\{Z(1)-Z(0)\geq K+1\}}] \to 0  \quad \text{as }K\to\infty,\\
    &\lim_{n\to \infty} n^{\theta_1}p_{n,n+k}=0 \quad \text{for each $ k\ge 1$,}    
\end{split}
\end{align}
where $\theta_1=\min_{i} \{\alpha_{i-1} - \alpha_{i-2}: 2\leq i \leq L\}$ as in Theorem \ref{prop:cycle}.
 Since $\theta_1\geq 0$, we have
\[
V(n+k)-V(n) \leq k(\theta_1+1)(n+k)^{\theta_1} \quad \text{and}\quad V(n-k)-V(n)\leq -k n^{\theta_1}
\]
for all $k,n\in\mathbb{N}$. This implies that there exist $\widetilde N'$ and a constant $c'$, which only depends on $\theta_1$, such that if $n\geq \widetilde N'$,
\begin{align*}
    \mathcal A^Z V(n) &= \E_{n}[(V(Z(1))-V(Z(0)))\,1_{\{Z(1)<Z(0)\}}]+\E_{n}[(V(Z(1))-V(Z(0)))\,1_{\{Z(1)-Z(0)\geq K+1\}}]  \\
    &\quad +\E_{n}[(V(Z(1))-V(Z(0)))\,1_{\{1\leq Z(1)-Z(0)\leq K\}}]  \\
    &\leq -c'n^{\theta_1}\E_{n}[(Z(0)-Z(1))\,1_{\{Z(1)<Z(0)\}}]+\E_{n}[(V(Z(1))-V(Z(0)))\,1_{\{Z(1)-Z(0)\geq K+1\}}]  \\
    &\quad + K(\theta_1+1)(n+K)^{\theta_1}\P_{n}(1\leq Z(1)-Z(0)\leq K),
\end{align*}
for all $K\ge 1$.
Hence, by the first inequality of \eqref{eq:increment of Z}, for $K\geq 1$ and $n\geq \max\{\widetilde{N},K,\widetilde N'\}$,
\begin{align}
    \mathcal A^Z V(n) 
    &\leq -c'C_{*} \,+\,\E_{n}[(V(Z(1))-V(Z(0)))\,1_{\{Z(1)-Z(0)\geq K+1\}}]  + K(\theta_1+1)(2n)^{\theta_1}\,\sum_{k=1}^Kp_{n,n+k}.
\end{align}
By the second inequality of \eqref{eq:increment of Z}, we can choose 
 $K$ large enough (depending only on $C_{*}$) so that the second term is less than $c'C_{*}/4$ for all $n\geq \widetilde{N}$. For this $K$, there exists $\widetilde{N}_K$ such that the last term  is less than $c'C_{*}/4$ for $n\geq \widetilde{N}_K$, by the last inequality of \eqref{eq:increment of Z}.
We have shown that \eqref{E:Lyapunov} holds with $c_0=c'C_{*}/2$ and $N_0=\max\{\widetilde{N},K,\widetilde{N}_K\}$.

Finally, we derive the desired result by using the discrete Dynkin's formula for $Z$ (see \cite[Section 4.3]{meyn1992stability}). Precisely,
for each $j>C$,
\begin{align}
    \mathbb E_j\left [V(Z(\tau^Z_{C})) \right ]
    &= V(j)  + \mathbb E_j \left [ \sum_{i=0}^{\tau^Z_{C}-1}\mathcal A^Z V(Z(i)) \right ]. \label{eq:discrete dynkins}
\end{align}

By \eqref{E:Lyapunov}, if $n\ge N_0>C$, then   
    \begin{align*}
        0\le \mathbb E_n\left [V(Z(\tau^Z_{C})) \right ]&=
        V(n)  + \mathbb E_{n} \left [ \sum_{i=0}^{\tau^Z_{C}-1}\mathcal A^Z V(Z(i)) \right ] \\
        &\le   n^{\theta} -c_0\mathbb E_n \left [\tau^Z_{C} \right ],      
    \end{align*}
    which completes the proof with $c=\frac{1}{c_0}$.
\end{proof}

To complete the proof of Proposition \ref{prop:mean first hitting}, 
 we establish the three statements in \eqref{eq:increment of Z}. 
\begin{lemma}\label{lem:transition prob of Z}
Under the same conditions in Proposition \ref{prop:mean first hitting},  let $X$ be the associated Markov chain. Let $Z$ be the 1-dimensional embedded Markov chain defined as \eqref{Def:Zi} with the probability $p^Z_{j,i}$ of the transition from state $i$ to state $j$ for $Z$. Then  
there exist $\widetilde{N}_1 \in \mathbb{N}$ and $ \widetilde{N}_2\in\mathbb{N}$ such that
\begin{align}
    &\inf_{n\geq \widetilde{N}_1} n^{\theta_1}\E_{n}[(Z(0)-Z(1))\,1_{\{Z(1)-Z(0)<0\}}] >0, \label{eq:appendix 1}\\
    & \sup_{n\geq \widetilde{N}_2}\E_{n}[(Z(1)^{\theta_1+1}-Z(0)^{\theta_1+1})\,1_{\{Z(1)-Z(0)\geq K+1\}}] \to 0  \quad \text{as }K\to\infty, \label{eq:appendix 2}\\
    &\lim_{n\to \infty} n^{\theta_1}p_{n,n+k}=0 \quad \text{for each $ k\ge 1$.}    \label{eq:appendix 3} 
\end{align}
\end{lemma}
\begin{proof}
\textbf{Step 1:}   We begin with showing \eqref{eq:appendix 1}. Let $x_A^{\eta^i}=\left(\sum_{j=1}^{|\eta^i|}\eta^i_j\right )_A$ be the accumulated changes of species $A$ by the path $\gamma_{\eta^i}$ such that
\begin{align}\label{eq:about d}
    x_A^{\eta^i}=\begin{cases}
      2\alpha_{i}-\alpha_{i+1}- \alpha_{i-1} = (\alpha_{i-1}-\alpha_{i-2})-(\alpha_i-\alpha_{i-1}) \quad &\text{if $2\le i<L$}\\
      \alpha_{L-1}-\alpha_{L-2}-(\alpha_1 -\alpha_0)& \text{if $i=L$}
    \end{cases}
\end{align}

Let $i\in I:=\{i=\displaystyle \argmin_{2\le i \le L}\{\alpha_{i-1}-\alpha_{i-2}\} \}$ be fixed. Suppose that $i\neq L$. Then, by considering the path $(n,0)+\gamma_{\eta^i}$ as a special case of all the possible paths from $(n,0)$ to $(n+d_i,0)$ for $X^D$, we have $p^Z_{n,n+d_i}\ge \mathbb  P_{(n,0)}(E_{\eta^i})=\prod_{j=1}^L q_j$, where
   \begin{align}
       &q_j= \label{eq:q_j}\\
       &\begin{cases}
           P\left (X^D(j)-X^D(j-1)=z_j-z_{j-1} \ \big | \ X^D(j-1)=(n+\alpha_{j-1},\beta_{j-1}) \right )  &\text{ if $1\le j\le i-1$}\\
           {}\\
            P\left (X^D(i)-X^D(i-1)=z_{i-1}-z_{i-2} \ \big | \ X^D(i-1)=(n+\alpha_{i-1}, \beta_{i-1}) \right ) &\text{ if $j=i$}\\
            {}\\
           P\left (X^D(j)-X^D(j-1)=z_j-z_{j-1} \ \big | \ X^D(j-1)=(n+\alpha_{j-1}+d_{i}, \beta_{j-1}) \right ) &\text{ if $i<j\le L$}
       \end{cases}\notag
   \end{align}
Then 
\begin{align}
    q_j&\ge P\left (\text{reaction $z_{j-1} \to z_j$ fires at $(n+\alpha_{j-1},\beta_{j-1})$ } \right ) =\frac{\lambda_{j-1}(n+\alpha_{j-1},\beta_{j-1})}{\sum_{k=0}^{j-1}\lambda_k(n+\alpha_{j-1},\beta_{j-1})} \quad \text{for $j<i$, and} \notag \\
    q_j&\ge P\left (\text{reaction $z_{i-1} \to z_i$ fires at $(n+\alpha_{i-1}+d_{j_0},\beta_{i-1})$ } \right ) =\frac{\lambda_{j-1}(n+\alpha_{i-1}+d_{j_0},\beta_{i-1})}{\sum_{k=0}^{j-1}\lambda_k(n+\alpha_{i-1}+d_{j_0},\beta_{i-1})} \quad \text{for $j>i$}. \notag 
\end{align}
Furthermore, 
\begin{align*}
    q_{i} &\ge P\left (\text{reaction $z_{j-1} \to z_j$ fires at $(n+\alpha_{i-1},\beta_{i-1})$ } \right ) =\frac{\lambda_{i-2}(n+\alpha_{i-1},\beta_{i-1})}{\sum_{k=0}^{i-1}\lambda_k(n+\alpha_{i-1},\beta_{i-1})}.
\end{align*}
Therefore we have $q_j=O(1)$ for $j\neq i$ and $q_i=O(1/n^{\alpha_{i-1}-\alpha_{i-2}})$ with respect to $n$. Therefore for any $i\in I \setminus \{L\}$, there exist $\widetilde N_1$ and $c_1>0$ such that if $n\ge \widetilde N_1$ then 
\begin{align}\label{eq:low bound of p^Z for T2}
    p^Z_{n,n+d_i} \ge \dfrac{c_1}{n^{\alpha_{i-1}-\alpha_{i-2}}}.
\end{align}
In the same way, if $i \in I\cap \{L\}$, we can show that \eqref{eq:low bound of p^Z for T2}.

Then as we showed in the proof of Lemma \ref{lem:non unique}, the accumulated change $ x_A^{\eta^i}<0$ if $i\in I$ by \eqref{eq:about d} by condition 1 of Assumption \ref{assump:cyclic}. Furthermore if $i\in I$, then $\alpha_{i-1}-\alpha_{i-2}=\theta_1$ by the definition of $\theta_1$
 Hence, $\lim_{n\to \infty}n^{\theta_1} p^Z_{n,n+d_i}>0$ by \eqref{eq:low bound of p^Z for T2} . Then we finally have that 
\begin{align*}
    &\lim_{n\to \infty} n^{\theta_1}\E_{n}[(Z(0)-Z(1))\,1_{\{Z(1)-Z(0)<0\}}] \ge 
    \inf_{n\geq \widetilde{N}_1} \sum_{\{i: d_i<0\}}n^{\theta_1}  \E_{n}[-d_i\ 1_{E_{\eta^{i}}}] \\
    &\ge \lim_{n\to \infty} \sum_{i\in I}n^{\theta_1}  \E_{n}[-d_i \ 1_{E_{\eta^{i}}}] \ge -\lim_{n\to \infty}  \sum_{i\in I} n^{\theta_1}d_{i}p^Z_{n,n+d_{i}}> 0.
\end{align*}
Therefore \eqref{eq:appendix 1} holds with $\widetilde N_1$.

  \noindent \textbf{Step 2:}    Now, we show \eqref{eq:appendix 2}. We define three events of $X^D$:
    \begin{align*}
        U^1_n&=\{\text{reaction $z_{i-1} \rightarrow z_{i}$ fires for some $1\le i\le L-1$ at some $x$ such that $x_A\ge n$ and $x_B\ge \beta_{L-1}$}\},\\
        U^2_n&=\{\text{reaction $z_{i-1} \rightarrow z_i$ fires for some $1\le i\le L-1$ at some $x$ such that $x_A\ge n$ and $x_B< \beta_{L-1}$}\},\\
        U^3_n&=\{\text{reaction $z_{L-1}\to z_{0}$ at some $x$ such that $x_A\ge n$}\}.        
    \end{align*}
    Note that there exist $\widetilde N'_2$ and $c_2>0$ such that if $n \ge \widetilde N'_2$, then
    \begin{align}\label{eq:U1}
        P(U^1_n)&\le \max_{\{x:x_A\ge n, x_B\ge \beta_{L-1}\}}\max_{0\le j\le L-2}\frac{\lambda_j(x)}{\sum^{L-1}_{i=1}\lambda_i(x)} \le \frac{c_2 }{n^{\alpha_{L-1}-\alpha_{L-2}}} \le\frac{c_2 }{n^{\theta_1}}.
    \end{align}
Let $k\ge 1$ be a fixed integer.  Let $M_1, M_2$, and $M_3$ be the numbers of the events $U^1_n, U^2_n$, and $U^3_n$ that have occurred within $[0,\nu^D_1)$ conditioned on $X^D(\nu^D_1)=(n+k,0)$, where $\nu^D_1=\inf\{i > 0:X^D(i)_B=0\}$.
There are important remarks pertaining to $M_i$'s. 
\begin{enumerate}
\item The only way to remove species $B$ is the reaction $z_{L-1}\to Z(0)$. Hence the last event must be $U^3_n$ right before $X$ hits the $A$-axis.
    \item   After the $U^3_n$ event fires at state $x$ such that $x_B\le \beta_{L-1}$ (actually it can only fire when $x_B=\beta_{L-1}$ in this case), $X^D$ will immediately hit the $A$-axis. Hence for $M_3$ times of the $U^3_n$ events before $X^D$ hits the axis, $B$ species has to be produced at least $M_3-1$ times when the number of $B$ species is bigger than or equal to $\beta_{L-1}$. This means that the $U^1_n$ events have to occur at least $M_3-1$ times. Consequently we have $M_1 \ge M_3 -1$.
   \item Between two $U^3_n$ events, at most $L-2$ times of $U^2_n$ events can occur. This is because after $L-2$ times of $U^2_n$ events without any $U^3_n$ events, $X^D_B$ is greater than $\beta_{L-1}$, hence $U^2_n$ events can no longer occur and the only way to decrease $X^D_B$ is the $U^3_n$ event. Hence $M_2 \le (L-2)(M_3-1)\le (L-2)M_1$. 
   \item Note that $k\le \rho M_1+\rho M_2-\alpha_{L-1}M_3\le \rho M_1+\rho(L-2)M_1$ where $\rho=\max\{\alpha_i-\alpha_{i-1}\}$. Hence $M_1\ge \lfloor \frac{k}{\rho(L-1
   )} \rfloor$.
\end{enumerate}
Let $k_\rho=\lfloor \frac{k}{\rho(L-1)} \rfloor$.
Consequently, if $M_1=m+ k_\rho$ for some $m\in \mathbb N$, then $M_2\le (L-2)(m+ k_\rho)$ and $M_3\le m+k_\rho+1$. Therefore
\begin{align}
    &p_{n,n+k}=P(X^D_{\nu^D_1}=(n+k,0)|X^D(0)=(n,0))  \notag\\
    &\le \sum_{m=0}^\infty \sum_{\ell=1}^{(L-2)(m+k_\rho)}\sum_{j=1}^{m+k_\rho+1} \begin{pmatrix}
        m+k_\rho+\ell+j \\
        m+k_\rho
\end{pmatrix}
\begin{pmatrix}
        \ell+j \\
        \ell
    \end{pmatrix}
P(M_1=m+k_\rho)
P(M_2=\ell)
P(M_3=j)  \notag \\
&\le \sum_{m=1}^\infty \sum_{\ell=1}^{(L-2)(m+k_\rho)}\sum_{j=1}^{m+k_\rho+1} \frac{(m+k_\rho+\ell+j)!}{(m+k_\rho)!\ell! j!} 
P(M_1=m+k_\rho).\label{eq:using U}
\end{align}
Using induction, for $\ell \le (L-2)(m+k_\rho)$ and $j\le m+k_\rho+1$, we can show that
$\dfrac{(m+k_\rho+\ell+j)!}{(m+k_\rho)!\ell! j!} \le \dfrac{(L(m+k_\rho)+1)!}{(m+k_\rho)! ((L-2)(m+k_\rho))! (m+k_\rho+1)!}$. That is, the combinatorial term is maximized with the largest possible values of $\ell$ and $j$. Using Stirling's formula we have that $\dfrac{(L(m+k_\rho)+1)!}{(m+k_\rho)! ((L-2)(m+k_\rho))! (m+k_\rho+1)!}\le c^{m+k_\rho}$ for some $c>1$ independent of $k$ and $m$. Hence using this and \eqref{eq:U1}, we can derive from \eqref{eq:using U} that there exist constants $c_3>0$ and  $\widetilde N'_2$  such that if $n>\tilde  N'_2$, then 
\begin{align*}
    p_{n,n+k}\le \sum_{m=1}^\infty (L-2)(m+k_\rho)(m+k_\rho+1) \frac{c^{m+k_\rho}c_2}{n^{\theta_1(m+k_\rho)}} \le \frac{c_3 c^{k_\rho}}{n^{\theta_1 k_\rho}},
\end{align*}
where we choose $\widetilde N'_2$ such that $(\widetilde N'_2)^{\theta_1}>c$ independently from $k$. 
Now, 
\begin{align}
  & \sup_{n\geq \widetilde{N}_2}\E_{n}[(Z(1)^{\theta_1+1}-Z(0)^{\theta_1+1})\,1_{\{Z(1)-Z(0)\geq K+1\}}] \notag \\
  &=   \sup_{n\geq \widetilde{N}_2} \sum_{k\ge K+1}\E_{n}[(Z(1)^{\theta_1+1}-Z(0)^{\theta_1+1})\,1_{\{Z(1)-Z(0)=k\}}]
   \notag \\
   &=\sup_{n\geq \widetilde{N}_2} \sum_{k\ge K+1}\left ((n+k)^{\theta_1+1}-n^{\theta_1+1} \right )p_{n,n+k} \notag\\
    &\le  \sum_{i=1}^{\theta_1+1} \begin{pmatrix}
    \theta_1+1\\
    i
\end{pmatrix}
\left( \sup_{n\geq \widetilde{N}_2} \sum_{k\ge K+1}
\frac{c_3 k^i c^{k_\rho}}{n^{\theta_1 k_\rho -(\theta_1+1-i)}} \right ) \notag \\
&\le \sum_{i=1}^{\theta_1+1} 
\begin{pmatrix}
    \theta_1+1\\
    i
\end{pmatrix}
\left( \sup_{n\geq \widetilde{N}_2} \sum_{k\ge K+1}
\frac{c_3 k^i c^{k}}{n^{\theta_1\frac{k}{\rho(L-1)}-(\theta_1+1-i)}} \right ). \label{eq:for the second one of appendix}
\end{align}
Finally  as $\dfrac{\theta_1}{\rho(L-1)}-\dfrac{(\theta_1+1-i)}{k}$ is increasing function of $k$, we can set $\widetilde N_2 > \tilde N'_2$ such that 
\begin{align*}
   \left( \frac{1}{n^{\frac{\theta_1}{\rho(L-1)}-\frac{(\theta_1+1-i)}{k}}} \right ) \le \frac{1}{2c}
\end{align*}
for any $k\ge 1$
if $n\ge \widetilde N_2$. Hence by \eqref{eq:for the second one of appendix}, 
\begin{align*}
 \sup_{n\geq \widetilde{N}_2}\E_{n}[(Z(1)^{\theta_1+1}-Z(0)^{\theta_1+1})\,1_{\{Z(1)-Z(0)\geq K+1\}}]   
 \le \sum_{i=1}^{\theta_1+1} 
\begin{pmatrix}
    \theta_1+1\\
    i
\end{pmatrix}
\left( \sup_{n\geq \widetilde{N}_2} \sum_{k\ge K+1}
\frac{c_3 k^i c^{k}}{(2c)^k} \right  ), 
\end{align*}
which goes to $0$, as $K\to \infty$. 

\noindent \textbf{Step 3:} Lastly for \eqref{eq:appendix 3}, we note that $ x_A^{\eta^i}$'s defined in \eqref{eq:about d} are negative if $i\in I$ (i.e. $\eta^i \in \mathcal T^2$). Hence
if $x\in J^{(2)}$, then $x_B < 0$ where $J^{(2)}$is as in Definition \ref{Def:J2}.  Then by \eqref{eq:reduction_assumption},  there exists $c_3>0$ such that $p_{n,n+k} \le \frac{c_3}{n^{\theta_1+1}}$ for any $k>0$. Hence \eqref{eq:appendix 3} follows.
\end{proof}

\begin{remark}
    The probability that we used for $q_j$ in \eqref{eq:q_j} is the transition probability given by $\alpha_{i-1}A+\beta_{i-1}B \to \alpha_i A+ \beta_i B$ after escape from the path $(n,0)+\gamma_{\eta^0}$. Here we used the fact that $\beta_i=i$ for each $i$. 
\end{remark}

\appendix

\section{Table of notations}\label{app:table}

\begin{center}
\vspace{0.2cm}
\begin{tabular}{|c|c|}
Symbol & Meaning \\
\hline
     $\S, \C, \Re$ & set of species, complexes and reactions, respectively ( Definition \ref{def:21}) \\ [1mm]
      $d$ & number of species\\ [1mm]
     $\Z^d$, $\R^d$ & $d$-dimensional integer and real vectors \\ [1mm]
    $\Z^d_{\ge 0}$ & $\{x \in \Z^d : x_i \ge 0 \text{ for each $i$} \}$\\ [1mm]
    $\R^d_{\ge 0}$ & $\{x \in \R^d : x_i \ge 0 \text{ for each $i$} \}$\\ [1mm]
    $\mathbb{I}_n$ & $\{ x\in \Z_{\ge 0}^d: x_d = 0, \|x\|_\infty = n\}$ where $\|x\|_\infty=\max_{1\leq i\leq d}x_i$ \\ [1mm]
    $\|x\|_{d-1,\infty}$ & $\max_{1\leq i\leq d-1}|x_i|$ (maximum among the first $d-1$ coordinates)  \\ [1mm]
    $\Lambda_n$ & $\{ x\in \Z_{\ge 0}^d:  \|x\|_{d-1,\infty} > n/2\}$ where $\|x\|_\infty=\max_{1\leq i\leq d-1}x_i$ \\ [1mm]
    $X_i(t)$ & the count of $i$ th species at time $t$ for the continuous time Markov chain  \\ [1mm]
    $X^D_i(k)$ & the count of $i$ th species after $k$-th transition of embedded chain\\ [1mm]
    $Z(i)$ & location of the $i$ th visit to the boundary face $\{x_d=0\}\cap \Z_{\ge 0}^d$ \eqref{Def:Zi} \\ [1mm]
    $q(x,z)$ & transition rate of continuous time Markov chain from $x$ to $z$ \\ [1mm]
    $\eta$ & a sequence of $d$-dimensional integer vectors with length $|\eta|$ \\ [1mm]
    $\gamma_\eta$ & a directed path starting at $0\in \Z^d$ that have increments $(\eta_i)_{i=1}^{|\eta|}$ \eqref{Def:gammaeta}\\ [1mm]
    $E_\eta$ & the event that $X^D$ follows $\eta$ for the first $|\eta|$ steps \eqref{eq:dominantcycle}\\ [1mm]
     $E_\eta^{\not= i}$ & the event that $X^D$ exits the path  $X^D(0)+\gamma_\eta$ at the $i$-th transition  \eqref{eq:E not i} \\ [1mm] 
    $\mathcal{T}^{(1)}, \mathcal{T}^{(2)}$ & sets of sequences of transitions in $\Z^d$ with finite lengths in Assumption \ref{assump}\\ [1mm]
    $\nu_i, \mu_i$ & $i$ th visit and exit time from the boundary face $\{x_d=0\}\cap \Z_{\ge 0}^d$ \eqref{eq:StoppingDef}\\ [1mm]
    $t_{\rm mix}^\delta(x)$ & Mixing time of $X$ starting at $x$ with threshold $\delta$ \eqref{E:mixing}\\ [1mm]
    $\tau_C$ &  first passage time of the continuous time Markov chain \eqref{Def:FPT}\\ [1mm]
    $\mathbb{P}_x$ & probability measure under which $X(0)=x$  \\ [1mm]
    $P^t(x,y)$ & probability of $X(t)=y$ under $\mathbb{P}_x$  \\[1mm]
    $p(\cdot,t)$ & probability density function of $Z(t)$ 
    \\[1mm]
    $\pi_x$ & stationary distribution of  $X$ (on the communication class containing $x$) \\ 
    $\alpha_i,\beta_i$ & coefficients for linear combinations in complexes in \eqref{eq:cyclicmodel} \\
    $\lambda_k$ & reaction intensity of $X$ associated with the $k$-th reaction \\
    $\kappa_i$ & rate constant associated with the $k$-th reaction \eqref{mass} \\
\end{tabular}
\end{center}

\section{Stationary distribution of stochastic reaction network \eqref{eq:cyclicmodel}}\label{app:stationary}
For a reaction network of a cyclic form as \eqref{eq:cyclicmodel}, the state space of the associated Markov chain under mass-action \eqref{mass} is a union of closed communication classes \cite{pauleve2014dynamical}. Hence the Markov chain associated with a reaction network of the form given in equation \eqref{eq:cyclicmodel} may have a state space that consists of multiple closed communication classes in general, instead of just one communication class. Therefore, to obtain the lower bound
of the mixing times of the reaction network using Corollary \ref{cor:mixing}, we need to show the property mentioned near Assumption \ref{A:Positiveecurrent}. 
In this section, for a fixed reaction network $(\Sp,\C,\Re)$ of a form of \eqref{eq:cyclicmodel}, we denote by $\mathbb S_n$ the closed communication class containing $(n,0)$.

\begin{lemma}\label{lem:non unique}
     Let $X$ be the associated continuous-time Markov process for a reaction system $(\Sp,\C,\Re)$ of the form given by \eqref{eq:cyclicmodel}. Suppose that $(\Sp,\C,\Re)$ is complex balanced \eqref{E:complexBalance} with the choice of the parameters $\kappa_i$'s.  If condition 1 of Assumption \ref{assump:cyclic} holds, then 
    there exists a unique stationary distribution $\pi_n$  on $\mathbb S_n$ for each $n\geq 1$ such that 
    $\displaystyle \lim_{n\to \infty} \pi_n(\Lambda_n)=0$ for $\Lambda_n = \{x \in\Z_{\ge 0}^2:\, x_A > \frac{n}{2}\} \cap \mathbb S_n$.
\end{lemma}
   \begin{proof}
       Note that due to complex balancing of $X$ \eqref{E:complexBalance}, there exists a unique stationary distribution $\pi_n$ on $\mathbb S_n$ such that
       \begin{align*}
           \pi_n(x_A,x_B)=\begin{cases}
               M_n\dfrac{c^x}{x_A! x_B!} \quad &\text{if $x\in \mathbb S_n$, and} \\
               0 &\text{otherwise},
           \end{cases}
       \end{align*}
       by Theorem \ref{thm:def0}, where $M_n$ is the normalizing constant. Note that if we show $\displaystyle \sup_{n}M_n < \infty$, then 
       \begin{align*}
           \dlim_{n\to \infty} \pi_n (\Lambda_n)\le \dlim_{n\to\infty} M_n\sum_{\{k\ge n/2 : x_A=k \text{ for some $x\in\mathbb S_n $} \}}\dfrac{c_A^k}{k!} \sum_{\ell=1}^\infty \frac{c_B^\ell}{\ell!}=0.
       \end{align*}
       To show $\displaystyle \sup_{n}M_n < \infty$, we will show the existence of a constant $C$ such that 
       \begin{align}\label{eq:min of the comm class}
           \sup_n \left ( \displaystyle \min_{x\in \mathbb S_n, x_B=0} x_A \right ) < C.
       \end{align}
       If such a $C$ exists, for any $n$ we can select a state $x\in \mathbb S_n$ satisfying that $x_A<C$ and $x_B=0$. Let such $x$ be denoted by $x(n)$. Then we can find $C'>0$ such that for any $n$, $\dfrac{c_A^{x(n)_A}c_B^{x(n)_B}}{x(n)_A! x(n)_B!} > C'$, which implies that 
       \begin{align*}
           \sup_{n}M_n = \sup_n \frac{1}{\sum_{x\in \mathbb S_n}\frac{c_A^{x_A}c_B^{x_B}}{x_A! x_B!}} \le   \sup_{n} \frac{1}{C'+\sum_{x\in \mathbb S_n\setminus\{x(n)\}}\frac{c_A^{x_A}c_B^{x_B}}{x_A! x_B!}} < \infty.
       \end{align*}
        In summary, to show that $\displaystyle \lim_{n\to \infty} \pi_n(\Lambda_n)=0$, it suffices to show \eqref{eq:min of the comm class}.

       Now let $x_A^{\eta^i}$ be the the accumulated change of species $A$ by the path $\gamma_{\eta^i}$ as \eqref{eq:about d}.     
Then for fixed $i_0\in I:=\left\{i=\displaystyle \argmin_{2\le i \le L}\{\alpha_{i-1}-\alpha_{i-2}\} \right\}$, the accumulated change $x_A^{\eta^{i_0}}\le 0$ but it must be  $x_A^{\eta^{i_0}}< 0$ by condition 1 in Assumption \ref{assump:cyclic}. For large enough $n$, thus, the state $(n-x_A^{\eta^{i_0}},0)$ is reachable form $(n,0)$ by the transitions in $\eta^{i_0}$. 

To achieve \eqref{eq:min of the comm class}, we will use the transitions $\eta^{i_0}$ multiple times in its order, so that there exists $C>0$ such that for any $n$, $x(n):=(x(n)_A,0)$ is reachable from $(n,0)$, where $x(n)_A$ is some integer such that $x(n)_A < C$. Note that the reachability from $(n,0)$ to $x(n)$ is up to whether the transitions in $\eta^{i_0}$ are available (i.e. non-zero transition rate) over the trajectories from $(n,0)$ to $x(n)$, which is constructed with the transitions in $\eta^{i_0}$. In other words, the reachability holds when there are enough amounts of $A$ and $B$.  Therefore we set  $C>\alpha_{L-1}$. Then when $x_A > C$,  for $x=(x_A,\beta_j)$, we have $\lambda_j(x)>0$ meaning that reaction $z_j\to z_{j+1}$ is available for any $i$. Then using the transitions in $\eta^{i_0}$ consecutively from $(n,0)$, the associated Markov chain is able to reach $x(n)$ implying \eqref{eq:min of the comm class}. We schematically describe this process with Figure \ref{fig:schematics}
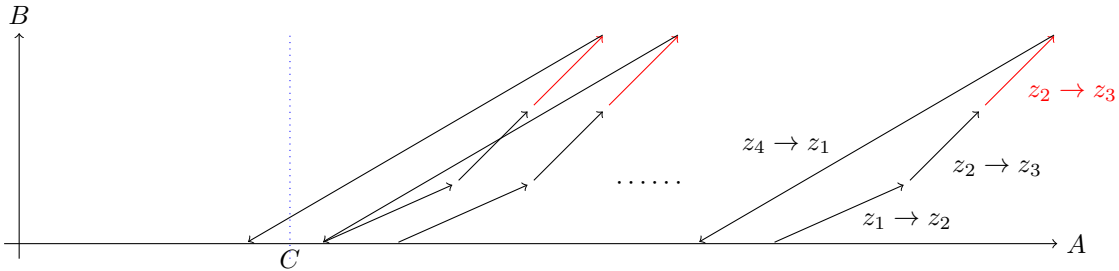
\begin{figure}[!h]
    \begin{equation*}
     \tikzset{state/.style={inner sep=1pt}}
   \begin{tikzpicture}[baseline={(current bounding box.center)}, scale=1]
   \node[state] (01) at (4.2,0.2)  {};
   \node[state] (02) at (6,1)  {};
   \node[state] (03) at (7,2)  {};
   \node[state] (04) at (8,3)  {};
   \node[state] (05) at (3.2,0.2)  {};
   \node[state] (1) at (5.2,0.2)  {};
   \node[state] (2) at (7,1)  {};
   \node[state] (3) at (8,2)  {};
   \node[state] (4) at (9,3)  {};
   \node[state] (5) at (4.2,0.2)  {};
   \node[state] (11) at (10.2,0.2)  {};
   \node[state] (12) at (12,1)  {};
   \node[state] (13) at (13,2)  {};
   \node[state] (14) at (14,3)  {};
   \node[state] (15) at (9.2,0.2)  {};
   \node[state] (20) at (8.6,1)  {\textbf{$\cdots \cdots$}};
   \node[state] (21) at (3.8,0)  {$C$};
   \node[state] (21) at (8.6,4)  {};
   \node[state] (31) at (12,0.5)  {$z_1\to z_2$};
   \node[state] (32) at (13.2,1.2)  {$z_2\to z_3$};
   \node[state] (33) at (14.2,2.2)  {\textcolor{red}{$z_2\to z_3$}};
      \node[state] (34) at (10.4,1.5)  {$z_4\to z_1$};
   \node[state] (35) at (9.2,0.2)  {};
   \path[->]
   (01) edge node {} (02)
    (02) edge node {} (03) 
    (03) edge[red] node {} (04)
    (04) edge node {} (05)
    (1) edge node {} (2)
    (2) edge node {} (3) 
    (3) edge[red] node {} (4)
    (4) edge node {} (5)
    (11) edge node {} (12)
    (12) edge node {} (13) 
    (13) edge[red] node {} (14)
    (14) edge node {} (15); 
   \draw[->] (0,0.2) -- (14,0.2) node[right] {$A$};
   \draw[->] (0.2,0) -- (0.2,3) node[above] {$B$};
   \draw[dotted] (3.8,0)[blue] -- (3.8,3) node[right] {};
  \end{tikzpicture}
\end{equation*}
    \caption{Schematics of the trajectories generated by the path $\eta^{i_0}$ when $i_0=2$. For $C>\alpha_{L-1}$, each reaction $z_j\to z_{j+1}$ used for the path $\eta^{i_0}$ is available at $x=(x_A,\beta_j)$.}
    \label{fig:schematics}
\end{figure}
\end{proof}

\section*{Acknowledgements}
\noindent This work was supported by National Science Foundation grants DMS-2152103 and DMS-2348164.

\bibliographystyle{plain}


\end{document}